\documentclass[11pt]{article}
\usepackage{listings}
\usepackage{pdflscape}
\usepackage{amsfonts}
\usepackage{epsfig}
\usepackage{lscape}
\usepackage{longtable}
\usepackage{amsmath,amssymb,amsthm}
\usepackage{graphicx}
\usepackage{subfigure}
\usepackage{color}
\usepackage{placeins}
\usepackage{url}
\usepackage{cases}
\usepackage{supertabular}
\usepackage{multirow}

\usepackage{amsmath}
\allowdisplaybreaks[4]
\usepackage{cite}
\usepackage{algorithmic}
\usepackage{algorithm,float}
\usepackage{booktabs}
\usepackage{enumerate}
\usepackage[online]{threeparttablex}
\usepackage{siunitx}
\usepackage{tikz}

\makeatletter
\newenvironment{breakablealgorithm}
{
	\begin{center}
		\refstepcounter{algorithm}
		\hrule height.8pt depth0pt \kern2pt
		\renewcommand{\caption}[2][\relax]{
			{\raggedright\textbf{\ALG@name~\thealgorithm} ##2\par}
			\ifx\relax##1\relax
			\addcontentsline{loa}{algorithm}{\protect\numberline{\thealgorithm}##2}
			\else
			\addcontentsline{loa}{algorithm}{\protect\numberline{\thealgorithm}##1}
			\fi
			\kern2pt\hrule\kern2pt
		}
	}{
		\kern2pt\hrule\relax
	\end{center}
}
\makeatother

\newtheorem{theorem}{Theorem}

\newtheorem{proposition}{Proposition}

\begin{document}
\title{\bf Estimation of sparse Gaussian graphical models with hidden clustering structure\footnotemark[1]}
\author{Meixia Lin\footnotemark[2], \quad Defeng Sun\footnotemark[3], \quad Kim-Chuan Toh\footnotemark[4], \quad Chengjing Wang\footnotemark[5]}

\date{April 16, 2020}
\maketitle

\renewcommand{\thefootnote}{\fnsymbol{footnote}}
\footnotetext[2]{Department of Mathematics, National University of Singapore, Singapore ({\tt lin\_meixia@u.nus.edu}).}
\footnotetext[3]{Department of Applied Mathematics, The Hong Kong Polytechnic University, Hung Hom, Hong Kong ({\ttfamily defeng.sun@polyu.edu.hk}). {This author is supported by Hong Kong Research Grant Council grant PolyU153014/18p and Shenzhen Research Institute
of Big Data, Shenzhen 518000 grant 2019ORF01002.}}
\footnotetext[4]{Department of Mathematics and Institute of Operations Research and Analytics, National University of Singapore, Singapore ({\tt mattohkc@nus.edu.sg}). The research of this author is partially supported
by the Academic Research Fund of the Ministry of Education of Singapore under grant number R-146-000-257-112.
}
\footnotetext[5]{Corresponding author. School of Mathematics, Southwest Jiaotong University, No. 999, Xian Road, West
Park, High-tech Zone, Chengdu 611756, China ({\tt renascencewang@hotmail.com}).}
\renewcommand{\thefootnote}{\arabic{footnote}}

\begin{abstract}
	Estimation of Gaussian graphical models is important in natural science when modeling the statistical relationships between variables in the form of a graph. The sparsity and clustering structure of the concentration matrix is enforced to reduce model complexity and describe inherent regularities. We propose a model to estimate the sparse Gaussian graphical models with hidden clustering structure, which also allows additional linear constraints to be imposed on the concentration matrix. We design an efficient two-phase algorithm for solving the proposed model. We develop a symmetric Gauss-Seidel based alternating direction method of the multipliers (sGS-ADMM) to generate an initial point to warm-start the second phase algorithm, which is a proximal augmented Lagrangian method (pALM), to get a solution with high accuracy. Numerical experiments on both synthetic data and real data demonstrate the good performance of our model, as well as the efficiency and robustness of our proposed algorithm.
\end{abstract}

\\[5pt]

\section{Introduction}
\label{sec: introduction}
Let $z\in \mathbb{R}^n$ be a random vector following a multivariate Gaussian distribution ${\cal N}(0,\Sigma)$ with an  unknown nonsingular covariance matrix $\Sigma$. Gaussian graphical models \cite{lauritzen1996graphical} estimate the concentration matrix $\Sigma^{-1}$ from a sample covariance matrix of $z$. It is known that $(\Sigma^{-1})_{ij}=0$ if and only if $z_i$ and $z_j$ are conditionally independent, given all the other variables. The Gaussian graphical model can be represented by an undirected graph ${\cal G}=({\cal V},{\cal E})$, where the vertices ${\cal V}$ contain  $n$ coordinates and the edges ${\cal E}=(e_{ij})_{1\leq i<j\leq n}$ describe the conditional independence relationships among $z_1,\cdots,z_n$. There is no edge between $z_i$ and $z_j$ if and only if $(\Sigma^{-1})_{ij}=0$.

To detect nonzero elements in the concentration matrix $\Sigma^{-1}$, researchers have proposed sparse Gaussian graphical models \cite{yuan2007model,banerjee2008model}. Given a sample covariance matrix $C\in{\mathbb S} ^n$, the sparse Gaussian graphical model attempts to estimate the concentration matrix $X^*:=\Sigma^{-1}$ by solving the following $\ell_1$-regularized log-likelihood minimization problem:
\begin{align}
\min_{X\succeq 0}  \displaystyle \ \Big\{\langle C, X\rangle - \log \det (X) + \rho \sum_{ i<j} |X_{ij}|\Big\},\label{sparse}
\end{align}
where $\rho$ is a given positive parameter, $\langle C, X\rangle$ is the standard trace inner product between $C$ and $X$, and $X\succeq 0$ means that $X\in \mathbb{S}^n$ is positive semidefinite. We adopt the convention that $\log 0:=-\infty$. The $\ell_1$-norm penalty, which is motivated by the lasso idea \cite{tibshirani1996regression}, enforces element-wise sparsity on $X$. There are many methods for solving the sparse Gaussian graphical model, such as the well-known GLasso algorithm \cite{friedman2008sparse}, the Newton-CG primal proximal point algorithm \cite{wang2010solving}, and QUIC \cite{hsieh2014quic}.

The concentration matrix may have additional structures other than sparsity. For example, Honorio et al. in \cite{honorio2009sparse} enforce the local constancy to find connectivities between two close or distant clusters of variables; H{\o}jsgaard and Lauritzen in \cite{hojsgaard2005restricted,hojsgaard2008graphical} propose the restricted concentration models where parameters associated with edges or vertices of the same class are restricted to being identical; Duchi et al. in \cite{duchi2012projected} penalize certain groups of edges together. In all these models, the clusters of the coordinates are assumed to be known. However, in many real applications like the gene expression in cancer data \cite{hughes2000functional,yu2017clustering}, the group/cluster information may be unknown in advance. The authors in \cite{marlin2009sparse} propose a two stage method for learning sparse Gaussian graphical models with unknown block structure. They propose a variational Bayes algorithm to learn the block structure in the first stage, then estimate the concentration matrix by using the block $\ell_1$ method in the second stage.

Here we aim to estimate the sparse concentration matrix and uncover the hidden clustering structure of the coordinates simultaneously. Note that in the context of a linear regression model where the regression coefficients are expected to be clustered into groups, the clustered lasso regularizer \cite{bondell2008simultaneous,lin2019efficient,petry2011pairwise,she2010sparse} has been widely used. We borrow the idea of the regularization term to discover the sparsity and unknown clustering structure in the Gaussian graphical models. Thus we modify the sparse Gaussian graphical model \eqref{sparse} as follows:
\begin{align}
\min_{X\succeq 0}  \displaystyle \ \Big\{\langle C, X\rangle - \log \det (X) +\rho \sum_{ i<j} |X_{ij}|+
\lambda \sum_{i<j} \sum_{s<t} |X_{ij} - X_{st}|\Big\},\label{uncon_P}
\end{align}
where $\rho,\lambda>0$ are given parameters. In the above model, the penalty on the pairwise differences is to force those entries of the concentration matrix associated with the same cluster of underlying random variables to be the same. In that way, the clustering structure of the random variables can then be discovered. In some more complicated cases, the conditional independence pattern may be partially known. To deal with those cases, one can  impose additional constraints on $X$ to get the following model:
\begin{align}
\min_{X\succeq 0}  \displaystyle \ \Big\{\langle C, X\rangle - \log \det (X) +\rho \sum_{ i<j} |X_{ij}|+
\lambda \sum_{i<j} \sum_{s<t} |X_{ij} - X_{st}| \bigm\vert  X_{ij}=0,\ (i,j)\in{\cal J}
\Big\},\label{pattern_P}
\end{align}
where ${\cal J}$ is the set of pairs of nodes $(i,j)$ such that $z_i$ and $z_j$ are known to be conditionally independent.

Motivated by the above discussions,  in this paper, we consider a more general problem which allows for general linear equality constraints to be imposed on $X$, i.e.,
\begin{equation}\label{P}
\min_{X\in{\mathbb S}^n} \displaystyle \  \Big\{   \langle C,X\rangle - \mu\log \det (X) +
\underbrace{\rho \sum_{ i<j} |X_{ij}|+
\lambda \sum_{i<j} \sum_{s<t} |X_{ij} - X_{st}|}_{Q(X)} \bigm\vert {\cal A} X = b,\  X\succeq 0\Big\},\tag{P}
\end{equation}
where ${\cal A}:{\mathbb S}^n\to \mathbb{R}^m$ is a given linear map, $b\in\mathbb{R}^m$ is a given vector, $\mu,\rho,\lambda>0$ are given parameters. Solving the problem \eqref{P} with a large $n$ is a challenging task due to the combination effects of the $n\times n$ positive semidefinite variable and the complicated regularization term together with the linear constraints. At a first glance, it would appear to be extremely expensive to evaluate the second part of $Q(X)$ as it involves approximately $n^4/8$ terms, thus it becomes unthinkable to even solve \eqref{P} for the case when $n$ is large. Fortunately, as we shall see later, the symmetric nature of the summation allows us to carry out the evaluation of the regularization term in $O(n^2\log n)$ operations. This reduction in the computation cost makes it possible to solve the problem \eqref{P} for large $n$.

Our contributions in this paper can be summarized in three parts. Firstly, we propose the model \eqref{P} to estimate the sparse Gaussian graphical model with hidden clustering structure, which also allows additional linear constraints to be imposed on the concentration matrix. As far as we are aware of, this is the first model that attempts to estimate the concentration matrix and uncover the hidden clustering structure in the variables simultaneously. Secondly, we design an efficient two-phase algorithm for solving the dual of \eqref{P}. We develope a symmetric Gauss-Seidel based alternating direction method of the multipliers (sGS-ADMM) to generate an initial point to warm-start the second phase algorithm, which is a proximal augmented Lagrangian method (pALM), to get a solution with high accuracy. For solving the pALM subproblems, we use the semismooth Newton method where the sparsity and clustering structure is carefully analysed and exploited in the underlying generalized Jacobians to reduce the computational cost in each semismooth Newton iteration. Thirdly, we conduct comprehensive numerical experiments on both synthetic data and real data to demonstrate the performance of our model, as well as the efficiency and robustness of our proposed algorithm. The numerical results show that our model can rather successfully estimate the concentration matrix as well as uncovering its clustering structure.

The remaining parts of the paper are organized as follows. In Section \ref{sec: problem_formulation}, we state the problem setup and some related results in the literature. In Section \ref{sec: algorithms}, we describe the proposed two-phase algorithm for solving our model. In Section \ref{sec: Numerical experiments}, we present the numerical results. Finally,  in Section \ref{sec: conclusion}, we make some concluding remarks.

Throughout the paper, we use ${\rm diag}(X)$ to denote a vector consisting of the diagonal entries of a matrix $X$ and ${\rm Diag}(x)$ to denote a diagonal matrix whose diagonal is given by a vector $x$. For any matrix $X\in \mathbb{R}^{n\times n}$, $\|X\|$ denotes the  Frobenius norm of $X$.

\section{Problem setup and preliminaries}
\label{sec: problem_formulation}
In this section, we set up the problem and present some related properties of the regularization term $Q(\cdot)$ and the function $\log\det(\cdot)$, respectively.

\subsection{Duality and optimality conditions}
\label{subsec: problem}
The minimization form for the dual of \eqref{P} is given by
\begin{align*} \label{D}
\min_{y\in\mathbb{R}^m,Z\in{\mathbb S}^n,S\in\mathbb{S}^{n} } & \displaystyle \   -\langle b,y\rangle  - \mu\log \det (Z) + Q^*(-S)-n\mu +n\mu \log\mu\tag{D} \\
\mbox{s.t.} &\quad  C- {\cal A}^* y - Z - S= 0, \quad Z \succeq 0,
\end{align*}
where ${\cal A}^*:\mathbb{R}^m\rightarrow \mathbb{S}^n$ is the adjoint map of ${\cal A}$, $Q^*$ is the Fenchel conjugate function of $Q$ that is defined by $Q^*(Y)=\sup\{\langle Y,X\rangle -Q(X)\mid X\in \mathbb{S}^n\}$ for any $Y\in \mathbb{S}^n$. The KKT system associated with \eqref{P} and  \eqref{D}  is given as
\begin{align}\label{KKT}
\left\{\begin{aligned}
&C- {\cal A}^* y - Z -  S= 0,\\
& XZ=\mu I_n, \ Z \succeq 0,\ X \succeq 0,\\
&0\in\partial Q(X)+S,\\
&{\cal A} X = b.
\end{aligned}\right.
\end{align}

Throughout this paper, we make the blanket assumption that ${\cal A}$ is surjective and the solution set to the KKT system \eqref{KKT} is nonempty. Since the objective function of \eqref{P} is strictly convex with respect to $X$, the optimal solution to \eqref{P} is unique, which we denote as $X^*$.

\subsection{The proximal mapping and Moreau envelope}
\label{subsec: moreauenvelope}
For a given closed convex function $f:{\cal H}\rightarrow \mathbb{R}$, where ${\cal H}$ is a finite dimensional real Euclidean space equipped with an inner product $\langle \cdot,\cdot\rangle$ and its induced norm $\|\cdot\|$. The Moreau envelope of $f$ at $x\in {\cal H}$ is defined as
\begin{align*}
{\rm E}_f(x)=\min_{y\in {\cal H}}  \displaystyle \  \Big\{ \frac{1}{2}\|y-x\|^2+f(y)\Big\}.
\end{align*}
The corresponding minimizer, which is called the proximal mapping of $f$ at $x$, is denoted as ${\rm Prox}_f(x)$. It is proved in \cite{moreau1965proximite,rockafellar1976monotone} that ${\rm Prox}_f(\cdot)$ is globally Lipschitz continuous with modulus $1$ and ${\rm E}_f(\cdot)$ is finite-valued, convex and continuously differentiable with
\begin{align*}
\nabla {\rm E}_f(x)=x-{\rm Prox}_f(x).
\end{align*}
The Moreau identity states that for any $t>0$, it holds that
\begin{align*}
{\rm Prox}_{tf}(x)+t{\rm Prox}_{f^*/t}(x/t)=x.
\end{align*}

\subsection{Results related ro the regularization term $Q(\cdot)$}
\label{subsec: clusteredlasso}
Let ${\cal B}:\mathbb{S}^n\rightarrow \mathbb{R}^{\bar{n}}$ be the linear map such that ${\cal B}X$ is the vector obtained from $X\in {\mathbb S}^n$ by concatenating the columns of the strictly upper triangular part of $X$ sequentially into a vector of dimension $\bar{n} := n(n-1)/2$. The adjoint ${\cal B}^*:\mathbb{R}^{\bar{n}}\rightarrow \mathbb{S}^n$ is such that ${\cal B}^*x$ is the operation of first putting the entries of the vector $x\in \mathbb{R}^{\bar{n}}$ into the strictly upper triangular part of an $n\times n$ matrix $X$, and then symmetrizing it. Denote
\begin{align*}
q(x) = \rho\| x\|_1 + \lambda p(x),\quad p(x) = \sum_{1\leq k < l \leq \bar{n}} |x_k-x_l|.
\end{align*}
Then it is obvious that
\begin{align*}
Q(X) = q({\cal B}X).
\end{align*}
The function $q(\cdot)$ is the clustered lasso regularizer in the context of the linear regression models,  which is studied in \cite{bondell2008simultaneous,lin2019efficient,petry2011pairwise,she2010sparse}. The associated conjugate function, proximal mapping and the corresponding generalized Jacobian of the proximal mapping has been carefully studied in \cite{lin2019efficient}. By making use of $q(\cdot)$, we have that for any $Y\in \mathbb{S}^n$,
\begin{align*}
Q^*(Y) =\sup_{X\in \mathbb{S}^n}\Big\{\sum_{i=1}^n X_{ii}Y_{ii}+2\langle {\cal B}X,{\cal B}Y\rangle -q({\cal B}X)\Big\}=\left\{\begin{aligned}
&q^*(2{\cal B}Y), && \mbox{if } {\rm diag}(Y)=0,\\
&+\infty,&& \mbox{otherwise,}
\end{aligned}
\right.
\end{align*}
and
\begin{align*}
{\rm Prox}_{Q}(Y) &=\underset{X\in \mathbb{S}^n}{\arg\min} \displaystyle \ \Big\{ \frac{1}{2}\|{\rm diag}(X)-{\rm diag}(Y)\|^2+\|{\cal B}X-{\cal B}Y\|^2+q({\cal B}X)
\Big\}\\
&={\rm Diag} \big({\rm diag}(Y)\big)+{\cal B}^* {\rm Prox}_q(2{\cal B}Y).
\end{align*}
Next we state the following proposition to compute $\partial{\rm Prox}_{Q}(Y)[H]$ for all $H\in \mathbb{S}^n$.
\begin{proposition}
	For any $Y\in \mathbb{S}^n$, it holds that
	\begin{align*}
	\partial{\rm Prox}_{Q}(Y)[H] = {\rm Diag} \big({\rm diag}(H)\big)+2{\cal B}^*\partial {\rm Prox}_q(2{\cal B}Y)[{\cal B}H],\quad \forall H\in \mathbb{S}^n,
	\end{align*}
	where $\partial{\rm Prox}_{Q}(Y)$ is the Clarke generalized Jacobian of ${\rm Prox}_Q(\cdot)$ at $Y$.
\end{proposition}
\begin{proof}
    The equality follows from \cite[Example 2.5]{hiriart1984generalized}.
\end{proof}

Here we present some results on the clustered lasso regularizer $p(\cdot)$ that are taken from \cite{lin2019efficient}. Denote ${\cal D}=\{x\in \mathbb{R}^{\bar{n}} \mid Bx\geq 0\}$, where $Bx=[x_1-x_2,x_2-x_3,\cdots,x_{\bar{n}-1}-x_{\bar{n}}]^T\in \mathbb{R}^{\bar{n}-1}$. Let $x^\downarrow$ be the vector whose components are those of $x$ sorted in a non-increasing order, that is $x^\downarrow_1\geq \cdots\geq x^\downarrow_{\bar{n}}$. Then we have the following proposition, which describes an efficient way for evaluating ${\rm Prox}_{q}(\cdot)$ .
\begin{proposition} \label{prop: proximalmapping_p}
	(a) For any $x\in \mathbb{R}^{\bar{n}}$, it can be proved that
	\[
	p(x)= \langle w,x^\downarrow \rangle,
	\]
	where $w\in \mathbb{R}^{\bar{n}}$ is defined by $w_k = \bar{n}-2k+1$, $k=1,\ldots,\bar{n}$. Thus the computational cost of evaluating $p(\cdot)$ can be reduced from $O(\bar{n}^2)$ to $O(\bar{n}\log \bar{n})$.
	\\[5pt]
	(b) For any given $y\in \mathbb{R}^{\bar{n}}$, let $P_y\in \mathbb{R}^{\bar{n}\times \bar{n}}$ be a permutation matrix such that $P_y y$ is sorted in a non-increasing order. Then the proximal mapping of $\lambda p$ at $y$ can be computed as
	\begin{align*}
	{\rm Prox}_{\lambda p}(y) = P_y^T \Pi_{\cal D}(P_yy-\lambda w),
	\end{align*}
	where $\Pi_{\cal D}(\cdot)$ (the metric projection onto ${\cal D}$) can be computed by the pool-adjacent-violators algorithm \cite{best1990active} in $O(\bar{n})$ operations.
	\\[5pt]
	(c) For any $y\in \mathbb{R}^{\bar{n}}$, the proximal mapping of $q$ at $y$ can be computed as
	\begin{align*}
	{\rm Prox}_{q}(y)
	= {\rm Prox}_{\rho\|\cdot\|_1} ({\rm Prox}_{\lambda p}(y)) = {\rm sign}({\rm Prox}_{\lambda p}(y))\circ \max(|{\rm Prox}_{\lambda p}(y)|-\rho,0).
	\end{align*}
	where the sign function, the absolute value and the maximum value are taken component-wise.
\end{proposition}

Next we consider the generalized Jacobian of ${\rm Prox}_q(y)$, which denoted as ${\cal M}(y)$. The detailed derivation of ${\cal M}(\cdot)$ could be found in \cite{lin2019efficient}. Note that ${\rm Prox}_q(\cdot)$ is strongly semismooth on $\mathbb{R}^{\bar{n}}$ with respect to ${\cal M}$. In the implementation of our proppsed algorithm, we need an explicitly computable element in ${\cal M}(y)$ for any given $y\in \mathbb{R}^{\bar{n}}$. As discussed in \cite{lin2019efficient}, we denote
\begin{align*}
{\cal I}_{\cal D}(y):=\{i\mid B_i \Pi_{\cal D}(y)=0,\ i=1,\cdots,\bar{n}-1\},
\end{align*}
where $B_i$ is the $i$-th row of $B$. Then we define two diagonal matrices $\Sigma={\rm Diag}(\sigma)\in \mathbb{R}^{(\bar{n}-1)\times(\bar{n}-1)}$ with
\begin{align*}
\sigma_i=\left\{
\begin{array}{ll}
1, & \mbox{if $i\in {\cal I}_{\cal D}(P_y y -\lambda w)$,}\\[5pt]
0, & \mbox{otherwise,}
\end{array}\right.\quad \mbox{for $i=1,2,\cdots,\bar{n}-1$,}
\end{align*}
and $\Theta={\rm Diag}(\theta)\in \mathbb{R}^{\bar{n}\times \bar{n}}$ with
\begin{align*}
\theta_i = \left \{
\begin{array}{ll}
0, & \mbox{if $|{\rm Prox}_{\lambda p}(y)|_{i} \leq \rho $,}\\[2mm]
1, & \mbox{otherwise,}
\end{array}\right.
\quad \mbox{for $i=1,2,\cdots,\bar{n}$.}
\end{align*}
Based on these notations, the following proposition provides a computable element in ${\cal M}(y)$.
\begin{proposition}\label{prop: jacobian_q}
For any $y\in \mathbb{R}^{\bar{n}}$, we have that
\begin{align*}
W=\Theta P_y^T (I_{\bar n}-B^T(\Sigma B B^T \Sigma)^{\dagger}B)P_y\in {\cal M}(y),
\end{align*}
where $(\cdot)^{\dagger}$ denotes the pseduoinverse. Further details on the computation of $W$ could be found in \cite[Proposition 2.8]{lin2019efficient}.
\end{proposition}

\subsection{Results related to the $\log\det(\cdot)$ function}
\label{subsec: logbarrier}
The following proposition states the computation of the proximal mapping of $-\mu\log\det(\cdot)$ and the corresponding Jacobian, which is directly obtained from \cite[Lemma 2.1]{wang2010solving}. For simplicity, we denote $r(X):=-\log\det(X)$ for any $X\succeq 0$.
\begin{proposition}\label{prop_Z}
	For any given $X\in \mathbb{S}^n$, with its eigenvalue decomposition $X=P{\rm Diag}(d)P^T$, where $d$ is the vector of eigenvalues and the columns of $P$ are the corresponding orthonormal set of eigenvectors. We assume that $d_1\geq \cdots\geq d_r>0\geq d_{r+1}\geq\cdots d_n$. Given $\mu>0$ and the scaler function $\phi_{\mu}^+(x):=(\sqrt{x^2+4\mu}+x)/2$ for all $x\in \mathbb{R}$, we define its matrix counterpart:
	\begin{align*}
	\phi_{\mu}^+(X):=P{\rm Diag}(\phi_{\mu}^+(d))P^T,
	\end{align*}
	where $\phi_{\mu}^{+}(d)\in \mathbb{R}^n$ is such that its $i$-th component is given by $\phi_{\mu}^{+}(d_i)$.
	\\[5pt]
	(a) The proximal mapping of $\mu r(\cdot)$ can be computed as
	\begin{align*}
	{\rm Prox}_{\mu r}(X) = \phi_{\mu}^+(X).
	\end{align*}
	(b) $\phi_{\mu}^+$ is continuously differentiable and its Fr\'echet derivative $(\phi_{\mu}^+)'(X)$ at $X$ is given by
	\begin{align*}
	(\phi_{\mu}^+)'(X)[H]=P(\Omega\circ (P^THP))P^T \quad \forall H\in \mathbb{S}^n,
	\end{align*}
	where $\Omega\in \mathbb{S}^n$ is defined by
	\begin{align*}
	\Omega_{ij}=\frac{\phi_{\mu}^+(d_i)+\phi_{\mu}^+(d_j)}{\sqrt{d_i^2+4\mu}+\sqrt{d_j^2+4\mu}},\quad i,j=1,\cdots,n.
	\end{align*}
\end{proposition}

\section{A two-phase algorithm}
\label{sec: algorithms}
In this section, we propose a two-phase algorithm to solve the problem \eqref{P} based on the augmented Lagrangian function of \eqref{D}. In Phase \uppercase\expandafter{\romannumeral1}, we design a symmetric Gauss-Seidel based alternating direction method of multipliers (sGS-ADMM) to solve the problem to a moderate level of accuracy. In Phase \uppercase\expandafter{\romannumeral2}, we employ a proximal augmented Lagrangian method (pALM) with its subproblems solved by the semismooth Newton method (SSN) to get a solution with high accuracy. Note that the sGS-ADMM not only can be used to generate a good initial point to warm-start the pALM, it can also be used alone to solve the problem. But as a first-order method, the sGS-ADMM may not be efficient enough in some cases to solve a problem to high accuracy. Thus in the second phase, we switch to the superlinearly convergent pALM to compute an accurate solution.

\subsection{Phase \uppercase\expandafter{\romannumeral1}: sGS-ADMM}
\label{subsec: sgsadmm}
A natural way to solve the problem \eqref{D} is the popular alternating direction method of the multipliers (ADMM), but as  shown via a counterexample in \cite{chen2016direct}, the directly extended sequential Gauss-Seidel-type multi-block ADMM may not be convergent even with a small step length. Thus,  in this paper, we employ a more delicate symmetric Gauss-Seidel-type multi-block ADMM, i.e., the sGS-ADMM to solve \eqref{D}. As is shown in \cite{chen2017efficient}, the sGS-ADMM is not only guaranteed to converge theoretically, in practice it also performs better than the possibly nonconvergent directly extended multi-block ADMM.

The Lagrangian function associated with \eqref{D} is given by	
\begin{align}
l(y,Z,S;X) &:= -\langle b,y\rangle  - \mu\log \det (Z) + Q^*(-S) + \delta_{\mathbb{S}_+^n}(Z)- n\mu+n\mu\log \mu \nonumber \\
&\quad - \langle C-{\cal A}^*y-Z-S,X\rangle,\quad \forall \,(y,Z,S,X)\in \mathbb{R}^m \times {\mathbb S}^n \times \mathbb{S}^{n} \times {\mathbb S}^n.\label{eq: lagragian function}
\end{align}
For $\sigma > 0$, the associated augmented Lagrangian function is
\begin{align}
{\cal L}_{\sigma}(y,Z,S;X) := l(y,Z,S;X) +\frac{\sigma}{2}\|C-{\cal A}^*y-Z-S\|^2.\label{eq: AL function}
\end{align}
Based on the augmented Lagrangian function \eqref{eq: AL function}, we design the sGS-ADMM for solving \eqref{D}. To be specific, we update $Z$ and $(y,S)$ alternatively as in the commonly used $2$-block ADMM, but with the key difference of applying the sGS iteration technique \cite{li2018qsdpnal} to the second block. The template for the algorithm is given as follows:
\begin{align*}
\left\{\begin{aligned}
&Z^{k+1}= \arg\min {\cal L}_{\sigma}(y^k,Z,S^k;X^k),\\
&\overline{y}^{k+1}= \arg\min {\cal L}_{\sigma}(y,Z^{k+1},S^k;X^k),\\
&S^{k+1}= \arg\min {\cal L}_{\sigma}(\bar{y}^{k+1},Z^{k+1},S;X^k),\\
&y^{k+1}= \arg\min {\cal L}_{\sigma}(y,Z^{k+1},S^{k+1};X^k),\\
&X^{k+1}=X^k-\tau\sigma(C-{\cal A}^*y^{k+1}-Z^{k+1}-S^{k+1}),
\end{aligned}\right.
\end{align*}
where $\tau\in(0,(1+\sqrt{5})/2)$ is a given step length that is typically set to be $1.618$. The implementation of updating each variable can be given as follows.

\paragraph{Updating of $Z$.} Given $\widehat{y},\widehat{S},\widehat{X}$, $\overline{Z}:=\arg\min {\cal L}_{\sigma}(\widehat{y},Z,\widehat{S};\widehat{X})$ can be obtained by
\begin{align*}
\overline{Z} =\underset{Z\succeq 0}{\arg\min} \displaystyle \  \Big\{\frac{\sigma}{2}\|Z+\frac{1}{\sigma}\widehat{M}\|^2- \mu\log\det(Z)
\Big\}=\phi_{\gamma}^+(-\frac{1}{\sigma}\widehat{M})=\frac{1}{\sigma}(\phi_{\gamma}^+(\widehat{M})-\widehat{M})
\end{align*}
where $\widehat{M} = \widehat{X}-\sigma(C-{\cal A}^*\widehat{y}-\widehat{S})$ and $\gamma=\mu\sigma$.

\paragraph{Updating of $y$.} Given $\widehat{Z},\widehat{S},\widehat{X}$, $\overline{y}:=\arg\min {\cal L}_{\sigma}(y,\widehat{Z},\widehat{S};\widehat{X})$ can be obtained by solving the linear system as
\begin{align*}
\overline{y} &= \underset{y\in \mathbb{R}^{n}}{\arg\min}\displaystyle \ \Big\{-\langle b,y\rangle+\frac{\sigma}{2}\|C-{\cal A}^*y-\widehat{Z}-\widehat{S}-\frac{1}{\sigma}\widehat{X}\|^2
\Big\}= ({\cal A}{\cal A}^*)^{-1} \big( {\cal A}(C-\widehat{S} -\widehat{Z} -\frac{1}{\sigma}\widehat{X}) +\frac{1}{\sigma}b \big).
\end{align*}

\paragraph{Updating of $S$.} Given $\widehat{y},\widehat{Z},\widehat{X}$, the updating of $S$ could be given as
\begin{align*}
\overline{S}&=\underset{S\in \mathbb{S}^{n}}{\arg\min}\displaystyle \  \Big\{Q^*(-S)+\frac{\sigma}{2}\|S+\widehat{V}\|^2 \Big\}= -{\rm Prox}_{Q^*}(\widehat{V})=-\widehat{V}+{\rm Prox}_Q(\widehat{V}),
\end{align*}
where $\widehat{V}=-(C-{\cal A}^*\widehat{y}-\widehat{Z}-\widehat{X}/\sigma)$.

The whole sGS-ADMM for solving \eqref{D} can be summarized as below.
\begin{breakablealgorithm}
	\caption{\small {\bf : sGS-ADMM}}
	\hspace*{0.02in} \raggedright {\bf Input:} $X^{0}\in {\mathbb S}^n_{++}$, $S^{0}\in \mathbb{S}^{n}$
	$y^0\in \mathbb{R}^{m}$, $\sigma > 0$, $\tau\in(0,(1+\sqrt{5})/2)$, $\gamma:=\mu\sigma$, and $k=0$.\\
	
	\begin{algorithmic}[1]
		\STATE   Compute
		\begin{align*}
		Z^{k+1} = (\phi_{\gamma}^+(M^k)-M^k)/\sigma,\quad  M^k=X^k-\sigma(C-{\cal A}^*y^k-S^k).
		\end{align*}
		
		\STATE Compute
		\begin{align*}
		\left\{\begin{array}{ll}
		\overline{y}^{k+1} = &({\cal A}{\cal A}^*)^{-1}\big( {\cal A}(C-S^k -Z^{k+1} -X^{k}/\sigma) +b/\sigma\big),\\[2mm]
		S^{k+1} = &-V^k+{\rm Prox}_Q(V^k), \quad V^k=-(C-{\cal A}^*\overline{y}^{k+1}-Z^{k+1}-X^k/\sigma),\\[2mm]
		y^{k+1} = &({\cal A}{\cal A}^*)^{-1}\big( {\cal A}(C-S^{k+1}-Z^{k+1}-X^{k}/\sigma)+b/\sigma\big).
		\end{array}\right.
	    \end{align*}
		
		\STATE Compute
		\begin{align*}
		X^{k+1}=X^{k}-\tau\sigma(C-{\cal A}^*y^{k+1}-S^{k+1}-Z^{k+1}).
		\end{align*}
		
		\STATE $k\leftarrow k+1$, go to Step 1.
	\end{algorithmic}
\end{breakablealgorithm}

The convergence result of the above algorithm can be obtained from \cite[Theorem 5.1]{chen2017efficient} without much difficulty.
\begin{theorem}
	Let $\{(y^k,Z^k,S^k,X^k)\}$ be the sequence generated by the sGS-ADMM. Then the sequence $\{y^k,Z^k,S^k\}$ converges to an optimal solution of \eqref{D} and $\{X^k\}$ converges to the optimal solution $X^*$ of \eqref{P}.
\end{theorem}

\subsection{Phase \uppercase\expandafter{\romannumeral2}: pALM}
\label{subsec: abcd}
The augmented Lagrangian method (ALM) is a widely used method for solving the convex optimization problem in the literature. It has the important property of possessing superlinear convergence guarantee.

We write the dual problem \eqref{D} in the following unconstrained form
\begin{align}\label{uncon_D}
\min_{y\in\mathbb{R}^m,S\in \mathbb{S}^n} \displaystyle \ \Big\{f(y,S):=&-\langle b,y\rangle  - \mu\log \det (C-{\cal A}^*y-S) + Q^*(-S)\notag\\
&+ \delta_{\mathbb{S}_+^n}(C-{\cal A}^*y-S)-n\mu +n\mu \log\mu
\Big\}.\tag{D$'$}
\end{align}
Denote
\begin{align*}
\widetilde{f}(y,S,Z,V)&=-\langle b,y\rangle  - \mu\log \det (C-{\cal A}^*y-S-Z) \\
&\quad + Q^*(-S+V)+ \delta_{\mathbb{S}_+^n}(C-{\cal A}^*y-S-Z) -n\mu +n\mu \log\mu.
\end{align*}
Then by \cite[Example 11.46]{rockafellar2009variational}, the Lagrangian function associated with \eqref{uncon_D} is
\begin{align*}
\widetilde{l}(y,S;X,U) &= \inf_{Z\in \mathbb{S}^n,V\in\mathbb{S}^n}\Big\{ \widetilde{f}(y,S,Z,V)-\langle Z,X\rangle -\langle U,V\rangle
\Big\}\\
&=-\langle b,y\rangle  - \langle C-{\cal A}^*y-S,X\rangle +\mu\log\det X -\delta_{\mathbb{S}_+^n}(X)-\langle U,S\rangle -Q(U).
\end{align*}
By \cite[Example 11.57]{rockafellar2009variational}, the corresponding augmented Lagrangian function is
\begin{align*}
&\widetilde{L}_{\sigma}(y,S;X,U)=\sup_{\widetilde{X}\in \mathbb{S}^n,\widetilde{U}\in\mathbb{S}^n}\Big\{\widetilde{l}(y,S;\widetilde{X},\widetilde{U}) -\frac{1}{2\sigma}\|X-\widetilde{X}\|^2-\frac{1}{2\sigma}\|U-\widetilde{U}\|^2
\Big\}\\
&=-\langle b,y\rangle-\frac{1}{\sigma}{\rm E}_{\mu\sigma r}(M(y,S))+\frac{1}{2\sigma}\|M(y,S)\|^2-\frac{1}{2\sigma}\|X\|^2-\frac{1}{\sigma}{\rm E}_{\sigma Q}(U-\sigma S)+\frac{1}{2\sigma}\|U-\sigma S\|^2-\frac{1}{2\sigma}\|U\|^2,
\end{align*}
where $M(y,S) = X-\sigma(C-{\cal A}^*y-S)$.

Based on the above notations, we describe the proximal augmented Lagrangian method (pALM) for solving \eqref{uncon_D} as follows.
\begin{breakablealgorithm}
	\caption{\small {\bf : pALM}}
	\hspace*{0.02in} \raggedright {\bf Input:} $y^0\in \mathbb{R}^m$, $X^{0}\in {\mathbb S}^n_{++}$, $S^0,U^{0}\in {\mathbb S}^{n}$, $\tau>0$, $\sigma_0 >0$, $k=0$.\\
	
	\begin{algorithmic}[1]
		\STATE   Compute
		\begin{align}
		(y^{k+1},S^{k+1})\approx \underset{y\in\mathbb{R}^m,S\in\mathbb{S}^n}{\arg\min}  \displaystyle \ \Big\{\Psi_k(y,S)
		:=\widetilde{L}_{\sigma}(y,S;X^k,U^k) +\frac{\tau}{2\sigma_k}(\|y-y^k\|^2+\|S-S^k\|^2)\Big\}.
		\label{subproblem:pALM}
		\end{align}
		
		\STATE Compute
		\begin{align*}
		X^{k+1}&= {\rm Prox}_{\mu\sigma_k r}(X^k-\sigma(C-{\cal A}^*y^{k+1}-S^{k+1})),\\
		U^{k+1}&= {\rm Prox}_{\sigma_k Q}(U^k-\sigma S^{k+1}).
		\end{align*}
		
		\STATE Update $\sigma_{k+1}\uparrow \sigma_{\infty}$, $k\leftarrow k+1$, go to Step 1.
	\end{algorithmic}
\end{breakablealgorithm}

\subsubsection{Convergence result of the pALM}
\label{subsubsec: convergencepALM}
The global convergence and global linear-rate convergence of the pALM can be obtained following the idea in \cite{li2019asymptotically}. To establish the convergence result, we define the maximal monotone operator
\begin{align*}
{\cal T}_{\widetilde{l}}(y,S,X,U):=\Big\{ (y',S',X',U')\mid (y',S',-X',-U')\in \partial \widetilde{l}(y,S;X,U)
\Big\},
\end{align*}
and its inverse operator
\begin{align*}
{\cal T}_{\widetilde{l}}^{-1}(y',S',X',U'):=\arg\min_{y,S}\max_{X,U}\ \displaystyle
\Big\{ \widetilde{l}(y,S;X,U)-\langle y',y\rangle-\langle S',S\rangle+\langle X',X\rangle+\langle U',U\rangle
\Big\}.
\end{align*}
As we note in the pALM, we need to specify the stopping criterion of computing the approximate solution $(y^{k+1},S^{k+1})$ in \eqref{subproblem:pALM}. Denote the operator
\begin{align*}
\Lambda={\rm Diag}(\tau I_m,\tau{\cal I}_n,{\cal I}_n,{\cal I}_n),
\end{align*}
where ${\cal I}_n$ is the identity operator over $\mathbb{S}^n$. We use the following stopping criteria for solving \eqref{subproblem:pALM}:
\begin{align}
\|\nabla \Psi_k(y^{k+1},S^{k+1})\|&\leq \frac{\min\{\sqrt{\tau},1\}}{\sigma_k}\varepsilon_k,\tag{A}\label{stopA}\\
\|\nabla \Psi_k(y^{k+1},S^{k+1})\|&\leq \frac{\min\{\sqrt{\tau},1\}}{\sigma_k}\delta_k\|(y^{k+1},S^{k+1},X^{k+1},U^{k+1})-(y^{k},S^{k},X^{k},U^{k})\|_{\Lambda},\tag{B}\label{stopB}
\end{align}
where $\{\varepsilon_k\}$ and $\{\delta_k\}$ are summable nonnegative sequences satisfying $\delta_k<1$ for all $k$.

Based on the above preparation, we could present the convergence result of the pALM in the following theorem, which is a direct application of \cite[Theorem 1 and Theorem 2]{li2019asymptotically}
\begin{theorem}
	(a) Let $\{(y^{k},S^{k},X^{k},U^{k})\}$ be the sequence generated by the pALM with the stopping criterion \eqref{stopA}. Then $\{(y^{k},S^{k},X^{k},U^{k})\}$ is bounded, $\{(y^{k},S^{k})\}$ converges to an optimal solution of \eqref{uncon_D}, and both $\{X^{k}\}$ and $\{U^{k}\}$ converge to the optimal solution $X^*$ of \eqref{P}.
	\\[5pt]
	(b) Let $\rho$ be a positive number such that $\rho>\sum_{k=0}^{\infty}\varepsilon_k$. Asuume that there exists $\kappa>0$ such that
	\begin{align*}
	{\rm dist}_{\Lambda}((y,S,X,U),{\cal T}_{\widehat{l}}^{-1}(0))\leq \kappa {\rm dist}(0,{\cal T}_{\widehat{l}}(y,S,X,U)),
	\end{align*}
	for all $(y,S,X,U)$ satisfying ${\rm dist}_{\Lambda}((y,S,X,U),{\cal T}_{\widehat{l}}^{-1}(0))\leq \rho$.
	Suppose that the initial point $(y^0,S^0,X^0,U^0)$ satisfies
	$$
	{\rm dist}_{\Lambda}((y^0,S^0,X^0,U^0),{\cal T}_{\widehat{l}}^{-1}(0))\leq \rho-\sum_{k=0}^{\infty}\varepsilon_k.
	$$
	Let $\{(y^{k},S^{k},X^{k},U^{k})\}$ be the sequence generated by the pALM with the stopping criteria \eqref{stopA} and \eqref{stopB}. Then for $k\geq 0$, it holds that
	\begin{align*}
	{\rm dist}_{\Lambda}((y^{k+1},S^{k+1},X^{k+1},U^{k+1}),{\cal T}_{\widehat{l}}^{-1}(0))\leq\mu_k
	{\rm dist}_{\Lambda}((y^{k},S^{k},X^{k},U^{k}),{\cal T}_{\widehat{l}}^{-1}(0)),
	\end{align*}
	where
	$$\mu_k=\frac{\delta_k+(1+\delta_k)\kappa \max\{\tau,1\} /\sqrt{\sigma_k^2+\kappa^2\max\{\tau^2,1\}}}{1-\delta_k}\rightarrow\mu_{\infty}:=\frac{\kappa \max\{\tau,1\}}{\sqrt{\sigma_{\infty}^2+\kappa^2\max\{\tau^2,1\}}}.$$
\end{theorem}

\subsubsection{A semismooth Newton method for solving the pALM subproblems}
\label{subsubsec: ssn}
As one can see, the main task in the pALM is to solve the subproblem \eqref{subproblem:pALM} in an efficient way. Note that given $(\widetilde{y},\widetilde{S},\widetilde{X},\widetilde{U})$, the subproblem \eqref{subproblem:pALM} has the form of
\begin{align*}
\min_{y\in \mathbb{R}^m,S\in \mathbb{S}^{n}}\ \displaystyle \ \Big\{ \Psi(y,S):=\widetilde{L}_{\sigma}(y,S;\widetilde{X},\widetilde{U}) +\frac{\tau}{2\sigma}(\|y-\widetilde{y}\|^2+\|S-\widetilde{S}\|^2)\Big\}.
\end{align*}
Since $\Psi(\cdot,\cdot)$ is a strongly convex function on $\mathbb{R}^m\times\mathbb{S}^{n}$, the above minimization problem has a unique optimal solution, denoted as $(\overline{y},\overline{S})$, which can be computed by solving the nonsmooth optimality condition:
\begin{align}
\nabla \Psi(y,S)=\begin{pmatrix}
-b+{\cal A}{\rm Prox}_{\mu\sigma r}(\widetilde{M}(y,S) )+\frac{\tau}{\sigma}(y-\widetilde{y})\\[5pt]
{\rm Prox}_{\mu\sigma r}(\widetilde{M}(y,S) )-{\rm Prox}_{\sigma Q}(\widetilde{U}-\sigma S)+\frac{\tau}{\sigma}(S-\widetilde{S})
\end{pmatrix}
=0,\label{eq: nablaphi}
\end{align}
where $\widetilde{M}(y,S) = \widetilde{X}-\sigma(C-{\cal A}^*y-S)$.

Define the operator $\hat{\partial}^2 \Psi(y,S):\mathbb{R}^m\times\mathbb{S}^{n}\rightarrow \mathbb{R}^m\times\mathbb{S}^{n}$ as
	\begin{align*}
	\hat{\partial}^2 \Psi(y,S)\begin{pmatrix}
	d_y\\[5pt]
	d_S
	\end{pmatrix}=\sigma \begin{pmatrix}
	{\cal A}\\[5pt]
	{\cal I}_n
	\end{pmatrix}(\phi_{\mu\sigma}^+(\widetilde{M}(y,S) ))'({\cal A}^*d_y+d_S)+\sigma \begin{pmatrix}
	0\\[5pt]
	\partial {\rm Prox}_{\sigma Q}(\widetilde{U}-\sigma S)[d_S]
	\end{pmatrix}+\frac{\tau}{\sigma} \begin{pmatrix}
	d_y\\[5pt]
	d_S
	\end{pmatrix},
	\end{align*}
	for any $d_y\in \mathbb{R}^m$, $d_S\in \mathbb{S}^{n}$. We can treat $\hat{\partial}^2 \Psi(y,S)$ as the generalized Jacobian of $\nabla \Psi(y,S)$ at $(y,S)$. By the analysis of the regularization term $Q(\cdot)$ and the function $r(\cdot)$ in Section \ref{sec: problem_formulation}, $\nabla \Psi(\cdot,\cdot)$ is strongly semismooth with respect to $\hat{\partial}^2 \Psi(\cdot,\cdot)$. Thus we could apply the semismooth Newton method (SSN) to solve \eqref{eq: nablaphi}, which has the following template.

\begin{breakablealgorithm}
	\caption{\small {\bf : SSN }}
	\hspace*{0.02in} \raggedright {\bf Input:} $\beta\in(0,1], \eta\in(0,1), \textrm{and } \zeta\in(0,\frac{1}{2}), \delta\in(0,1)$, choose $y^{0}\in \mathbb{R}^m$, $S^0\in\mathbb{S}^n$, and set $j=0$.\\
	
	\begin{algorithmic}[1]
		\STATE Choose ${\cal H}_j\in \partial {\rm Prox}_{\sigma Q}(\widetilde{U}-\sigma S^j)$, use the conjugate gradient method (CG) to solve the linear system
		\begin{equation*}
        \sigma \begin{pmatrix}
		{\cal A}\\[5pt]
		{\cal I}_n
		\end{pmatrix}(\phi_{\gamma}^+(\widetilde{M}(y^j,S^j)))'({\cal A}^*d_y^j+d_S^j)+\sigma \begin{pmatrix}
		0\\[5pt]
		{\cal H}_j d_S^j
		\end{pmatrix}+\frac{\tau}{\sigma} \begin{pmatrix}
		d_y^j\\[5pt]
		d_S^j
		\end{pmatrix}=
		-\nabla\Psi(y^j,S^j)\label{eq-Newton}
		\end{equation*}
		to obtain $d_y^j$ and $d_S^j$ such that the residual is no larger than $\min\{\eta,\|\nabla\Psi(y^j,S^j)\|^{1+\beta}\}$.
		
		\STATE  Set $\alpha_{j}=\delta^{m_{j}}$, where $m_{j}$ is the first nonnegative integer $m$ for which
		$$\Psi(y^{j}+\delta^{m}d_y^{j},S^{j}+\delta^{m}d_S^{j})\leq
		\Psi(y^{j},S^j)+\zeta\delta^{m} \biggl\langle \nabla\Psi(y^{j},S^j),\begin{pmatrix}
		d_y^j\\[5pt]
		d_S^j
		\end{pmatrix}\biggr\rangle.$$		
		
		\STATE Set $y^{j+1}=y^{j}+\alpha_{j}d_y^{j}$, $S^{j+1}=S^{j}+\alpha_{j}d_S^{j}$, $j\leftarrow j+1$, go to Step 1.
	\end{algorithmic}
\end{breakablealgorithm}

Since the operator $\hat{\partial}^2 \Psi(\cdot,\cdot)$ is positive definite, we can easily obtain the following superlinear convergence result of the SSN method from \cite{zhao2010newton}.
\begin{theorem}
	Let $\{(y^j,S^j)\}$ be the sequence generated by the SSN method, then $\{(y^j,S^j)\}$ converges to $(\overline{y},\overline{S})$ and
	\begin{align*}
	\|(y^{j+1},S^{j+1})-(\overline{y},\overline{S})\|=O(\|(y^j,S^j)-(\overline{y},\overline{S})\|^{1+\beta}).
	\end{align*}
\end{theorem}

\section{Numerical experiments}
\label{sec: Numerical experiments}
In this section, we present some numerical experiments on both synthetic and real data to demonstrate the performance of the proposed model and the efficiency of the two-phase algorithm. In our algorithm, we fix the iteration number of the sGS-ADMM in Phase \uppercase\expandafter{\romannumeral1} to be $200$. As we discuss before, in some cases, the sGS-ADMM alone may not be efficient enough to solve the problems. To deal with these cases, we additionally apply the the pALM, which is more complicated to implement. Since up to our knowledge, there is no other existing algorithm in the literature which is suitable to solve \eqref{P} for large $n$, we compare our algorithm with the sGS-ADMM alone to demonstrate the efficiency and robustness of our two-phase algorithm. All experiments are implemented in {\sc Matlab} 2018b on a windows workstation (12-core, Intel Xeon E5-2680 @ 2.50GHz, 128 G RAM).

\subsection{Stopping criteria}
In our experiments, we measure the infeasibilities of the primal and dual problems by $R_{P},R_{D}$, and measure the complementarity condition by $R_{C}$, where
\begin{align*}
R_{P} :=& \frac{\|{\cal A}X-b\|}{1+\|b\|}, \quad R_{D} :=\frac{\|C-{\cal A}^*y-S-Z\|}{1+\|C\|},\\
R_{C} :=& \max\Big\{\frac{\|XZ-\mu I_n\|}{1+\|X\|+\|Z\|},\frac{\|X-{\rm Prox}_{Q}(X-S)\|}{1+\|X\|+\|S\|}
\Big\}.
\end{align*}
Note that in Phase \uppercase\expandafter{\romannumeral2}, the variable $Z$ could be constructed according to the derivation of the Lagrangian function as $(\phi_{\gamma}^+(M)-M)/\sigma$, where $M=X-\sigma(C-{\cal A}^*y-S)$. We stop the algorithm when
\begin{equation*}
\max\{R_{P},R_{D},R_{C}\} < {\tt Tol},
\end{equation*}
with ${\tt Tol} = 10^{-6}$ as the default. We also stop the algorithm if it reaches the maximum iteration number, $200$ for the pALM and $50000$ for the sGS-ADMM. Furthermore, we also use the relative gap to measure the quality of the solution, which is defined as
\begin{equation*}
R_{G} := \frac{|{\tt pobj}-{\tt dobj}|}{1+|{\tt pobj}|+|{\tt dobj}|},
\end{equation*}
where ${\tt pobj}$ and ${\tt dobj}$ are the primal and dual objective function values given by
\begin{align*}
{\tt pobj}&=\langle C,X\rangle - \mu\log \det (X) + Q(X),\quad {\tt dobj}=\langle b,y\rangle  + \mu\log \det (Z) +n\mu -n\mu \log\mu.
\end{align*}

\subsection{Experimental settings}
In each experiment, we are given $p$ samples $\{z_i\}_{i=1}^p$ with $z_i\in \mathbb{R}^n$, the sample covariance matrix $C$ is constructed as
\begin{align*}
C = \frac{1}{p}\sum_{i=1}^p(z_i-\bar{z})(z_i-\bar{z})^T,\quad \bar{z} = \frac{1}{p}\sum_{i=1}^p z_i.
\end{align*}
For the parameters in the model \eqref{P}, we take
\begin{align}
\mu = 1,\quad \lambda = k\rho /\bar{n}.\label{eq: parameter}
\end{align}
Thus in each experiment when estimating the Gaussian graphical model, we need to determine the approximate values of $\rho$ and $k$, which balance the sparsity and clustering structure. The constraint data ${\cal A}$ and $b$ is discussed individually in each experiment.

\subsection{Experiments on synthetic data}
In this subsection, we conduct experiments on artificial datasets on covariance selection and graph recovering. The first experiment is on the covariance selection problem where the true concentration matrix is constructed to have sparsity and clustering structure. Since our model can also be used to recover the graph structure, we also create several synthetic datasets based on different graph-based models constructed via the procedure in \cite{egilmez2017graph}. The sample size $p$ is fixed to be $10n$. We test the case when the constraint takes the form as the model \eqref{pattern_P}, that is, the sparsity pattern is partially known. The set ${\cal J}$ is generated following the idea in \cite{lu2009smooth} as
\begin{align*}
{\cal J}=\{(i,j)\mid (\Sigma^{-1})_{ij}=0,\ |i-j|\geq 5\},
\end{align*}
where $\Sigma^{-1}$ is the true concentration matrix. In order to measure the experimental performance, we adopt two metrics used in \cite{egilmez2017graph,kumar2020unified}. The first one is the relative error between $\Sigma^{-1}$ and $X^*$:
\begin{align*}
{\rm RE}:=\frac{\|X^*-\Sigma^{-1}\|}{\|\Sigma^{-1}\|}.
\end{align*}
The other one is the F-score metric:
\begin{align*}
{\rm FS}:=\frac{2{\rm tp}}{2{\rm tp}+{\rm fn}+{\rm fp}},
\end{align*}
where true positive (${\rm tp}$) stands for the case when the computed solution $X^*$ detects an edge correctly, false negative (${\rm fn}$) means that $X^*$ misses an edge and false positive (${\rm fp}$) stands for the case when $X^*$ detects an edge which should not be present. Note that a F-score value of $1$ means perfect recovery of the sparsity pattern of the concentration matrix.

\paragraph{Synthetic dataset \uppercase\expandafter{\romannumeral1}: covariance selection.} We first generate a $0$-$1$ matrix in $\mathbb{S}^n$ denoted as $P$, where $1$ represents the position of the non-zero elements of the concentration matrix. Since we focus on the sparse Gaussian graphical model with clustering structure, we generate $P$ according to the model $ {\cal P}(n,n_G,p^{\rm big},p^{\rm small},p^{\rm mid})$, where $n_G$ means the number of clusters of the coordinates, $p^{\rm big}$, $p^{\rm mid}$ are the probabilities of having an edge between the coordinates within and across the clusters, $p^{\rm small}$ is the probability of having edges between two different clusters. In the model, the number of coordinates in each cluster are randomly chosen. For simplicity, the coordinates are sorted according to the clusters. Based on the sparse pattern of $P$, we generate the random concentration matrix modified from the procedure in \cite{d2008first,wang2010solving}. Let $\Sigma^{-1}$ be a matrix which has the same sparsity structure as $P$, but uniformly distributed random entries on $[-1,1]$. To ensure that the positive definiteness of $\Sigma^{-1}$, we compute
\begin{align*}
\Sigma^{-1} = \Sigma^{-1}+I_n,\quad \Sigma^{-1} = \Sigma^{-1}+\max\{-1.2\min({\rm eig}(\Sigma^{-1})),0.001\}I_n.
\end{align*}
For each test problem, we sample $p=10n$ instances from the multivariate Gaussian distribution ${\cal N}(0,\Sigma)$, and fix $p^{\rm big}=0.8$, $p^{\rm small}=0.2$, $p^{\rm mid}=0.5$. In this experiment we fix $k=1$ in \eqref{eq: parameter}.

To visualize the estimation performance of our model, we refer to Figures \ref{fig: rand1_500_0001} and \ref{fig: rand1_500_0005}, which show the estimated result for the case when $(n,n_G)=(500,10)$ with two different parameters. In the figures, the input sparsity pattern shows the pattern of $C^{-1}$. As we can see in the figures, our model could recover the sparsity and clustering structure of the unknown concentration matrix with the small sample size of $p=10n$ in this experiment. Table \ref{table_rand1_acc} reports the relative errors and F-scores for different problem instances. The performance is satisfactory considering the small sample size and complicated structure. As one can observe from Figures \ref{fig: rand1_500_0001} and \ref{fig: rand1_500_0005}, the estimated sparsity pattern of the concentration matrix closely reflects the true sparsity pattern. Note that we are able to solve a very large instance with matrix dimension $n=4000$ and $3579004$ linear constraints in 18 minutes and 23 seconds. Generally, our proposed algorithm works quite well as shown in Table \ref{table_rand1_time}. As one can see, the test problems in this case are all solved to the desired accuracy by the sGS-ADMM in Phase \uppercase\expandafter{\romannumeral1}.

\begin{figure}[H]
	\centering
	\vspace{-0.2cm}
	\includegraphics[width=6.5in]{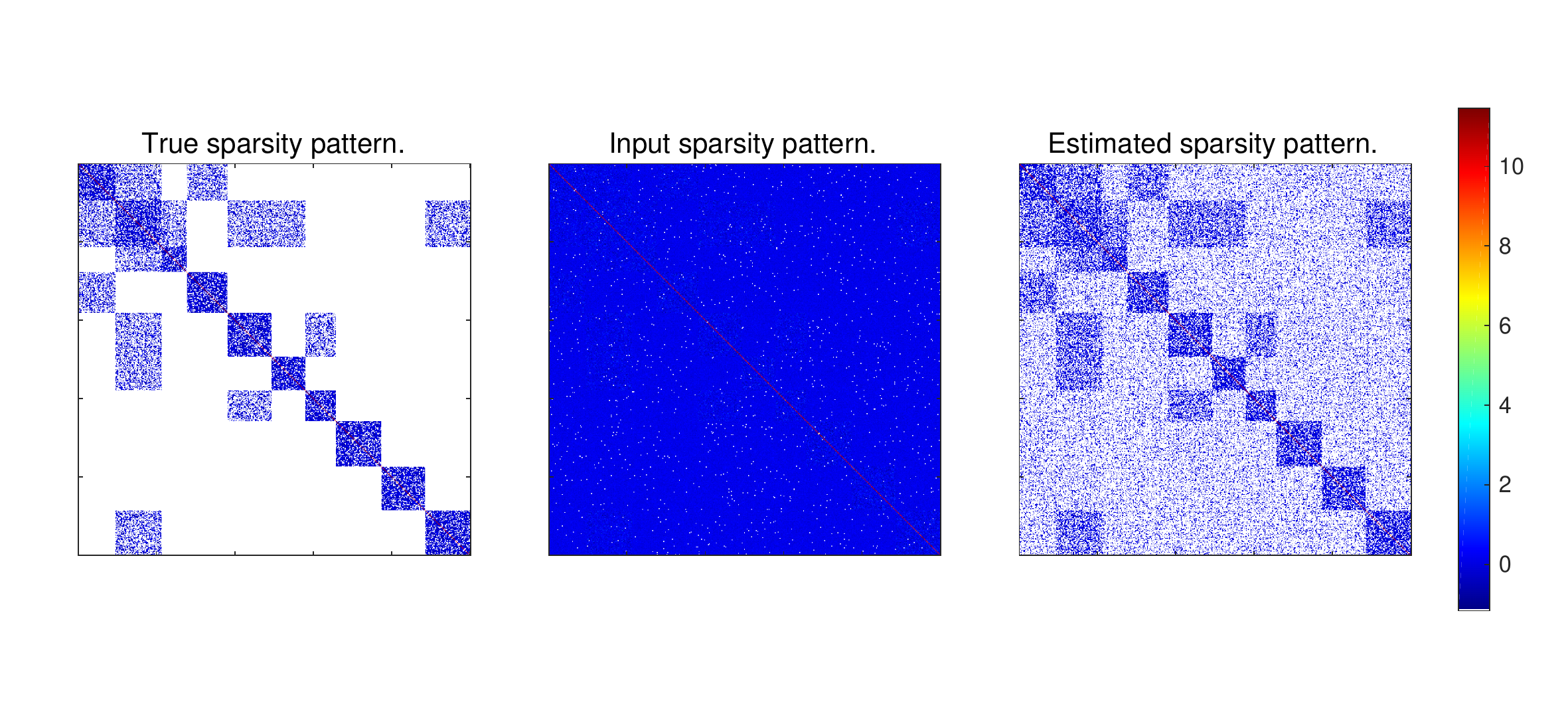}
	\vspace{-1.5cm}
	\caption{The estimated result of covariance selection for $(n,n_G)=(500,10)$ and $\rho=0.001$.}
	\label{fig: rand1_500_0001}
\end{figure}

\begin{figure}[H]
	\centering
	\vspace{-0.8cm}
	\includegraphics[width=6.5in]{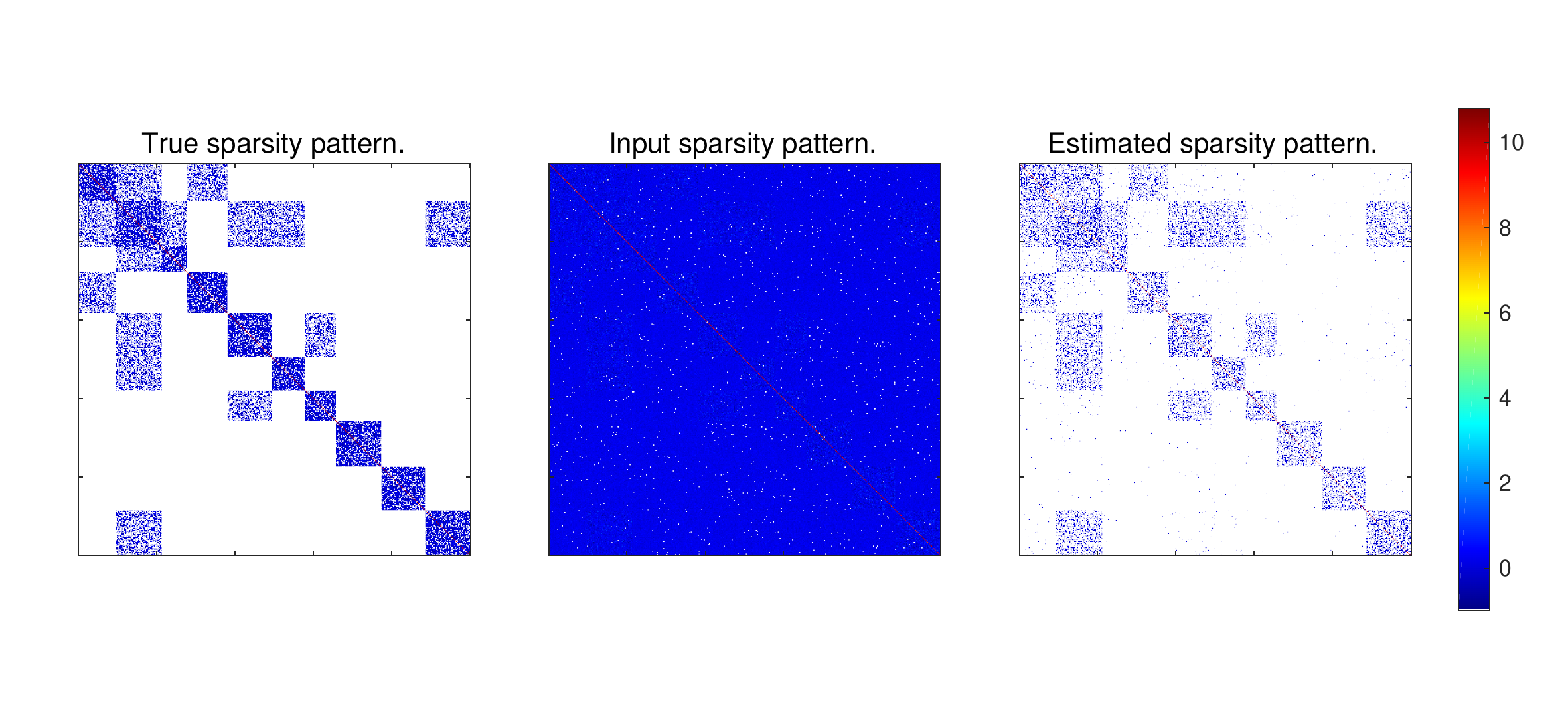}
	\vspace{-1.5cm}
	\caption{The estimated result of covariance selection for $(n,n_G)=(500,10)$ and $\rho=0.005$.}
	\label{fig: rand1_500_0005}
\end{figure}

{\small
	\begin{ThreePartTable}
		\renewcommand{\multirowsetup}{\centering}
		\renewcommand\arraystretch{1.5}
		\begin{longtable}{cccc}
			\caption{Performance of the proposed model for covariance selection.} \label{table_rand1_acc}\\[-6pt]
			\hline
			\hline
			$(n,n_G,m)$ &  $\rho$  & RE & FS  \\
			\hline
			\endfirsthead
			
			\multicolumn{4}{c}{{ \tablename\ \thetable{} -- continued from previous page}} \\
			
			\hline
			\hline
			$(n,n_G,m)$ &  $\rho$  & RE & FS  \\
			\hline
			\endhead
			
			\multicolumn{4}{r}{{Continued on next page}} \\
			\hline
			\hline
			\endfoot
			
			\hline
			\hline
			\endlastfoot	
			
			(500,10,52623) & 0.001 & 1.53e-1  & 5.06e-1\\
			\hline
			
			(500,10,52623) & 0.005 & 2.55e-1 & 6.16e-1\\
			\hline
			
			(1000,20,220141) & 0.001 & 1.41e-1 & 5.22e-1\\
			\hline
			
			(1000,20,220141) & 0.005 & 3.09e-1 & 4.81e-1\\
			\hline
			
			(2000,50,894283) & 0.001 & 1.69e-1 & 6.62e-1\\
			\hline
			
			(3000,50,2012956) & 0.001 & 2.12e-1 & 7.45e-1\\
			\hline
			
			(4000,50,3579004) & 0.001 & 2.56e-1 & 6.57e-1\\
		\end{longtable}
	\end{ThreePartTable}
}

{\small
	\begin{ThreePartTable}
		\renewcommand{\multirowsetup}{\centering}
		\renewcommand\arraystretch{1.5}
		\begin{longtable}{cc|cc|cc|c}
			\caption{Performance of our two-phase algorithm for solving covariance selection problems.} \label{table_rand1_time}\\[-6pt]
			\hline
			\hline
			$(n,n_G,m)$ &  $\rho$  & $\max\{R_{P},R_{D},R_{C}\}$   & $R_G$ & \multicolumn{2}{c|}{ Iteration}  & Time \\
			&&&& Phase \uppercase\expandafter{\romannumeral1} & Phase \uppercase\expandafter{\romannumeral2} & \\
			\hline
			\endfirsthead
			
			\multicolumn{7}{c}{{ \tablename\ \thetable{} -- continued from previous page}} \\
			
			\hline
			\hline
			$(n,n_G,m)$ &  $\rho$  & $\max\{R_{P},R_{D},R_{C}\}$   & $R_G$ & \multicolumn{2}{c|}{ Iteration}  & Time \\
			&&&& Phase \uppercase\expandafter{\romannumeral1} & Phase \uppercase\expandafter{\romannumeral2} & \\
			\hline
			\endhead
			
			\multicolumn{7}{r}{{Continued on next page}} \\
			\hline
			\hline
			\endfoot
			
			\hline
			\hline
			\endlastfoot	
			
			(500,10,52623) & 0.001 & 6.94e-7 & 3.64e-7 & 71 & -- & 00:00:07\\
			\hline
			
			(500,10,52623) & 0.005 & 6.67e-7 & 2.38e-7 & 62 & -- & 00:00:07\\
			\hline
			
			(1000,20,220141) & 0.001 & 8.52e-7 & 4.13e-7 & 75 & -- & 00:00:35\\
			\hline
			
			(1000,20,220141) & 0.005 & 9.57e-7 & 3.11e-7 & 65 & -- & 00:00:30\\
			\hline
			
			(2000,50,894283) & 0.001 & 6.68e-7 & 2.88e-7 & 84 & -- & 00:03:13\\
			\hline
			
			(3000,50,2012956) & 0.001 & 8.56e-7 & 3.14e-7 & 90 & -- & 00:08:25\\
			\hline
			
			(4000,50,3579004) & 0.001 & 8.70e-7 & 2.95e-7 & 96 & -- & 00:18:23\\
		\end{longtable}
	\end{ThreePartTable}
}

\paragraph{Synthetic dataset \uppercase\expandafter{\romannumeral2}: grid graph recovery.}
We consider a grid graph denoted as ${\cal G}_{\rm grid}(n)$, where $n=t^2$ is the number of nodes. In the grid graph, each node is attached to their four nearest neighbours except for the vertices at the boundary. The edge weights are randomly selected based on a uniform distribution from $[0.1,3]$. We sample $p=10n$ instances from the multivariate Gaussian distribution ${\cal N}(0,L^{\dagger})$, where $L$ is the associated Laplacian matrix of the graph. We fix $\rho=0.01$ and $k=2$ in \eqref{eq: parameter}.

Figures \ref{fig: randgrid_64} and \ref{fig:randgrid_64_com} show the estimation result of the grid graph recovery problem for the case $n=64$, where the visualization of the graphs are constructed via the corresponding adjacency matrices. Table \ref{table_randgrid_acc} displays the two metrics (RE and FS) of the estimated results obtained by our model. Note that the metrics for the case $n=64$ are comparable to the results in \cite{kumar2020unified}. As we can see from the figures, the estimated sparsity pattern of the grid graph closely matches the true pattern. In the visualization, a darker edge means that the corresponding $X^*_{ij}$ has a larger (in magnitude) negative value in the computed concentration matrix $X^*$. The numrical performance of the two-phase algorithm and the sGS-ADMM are reported in Table \ref{table_randgrid_time}. We can see from the table that for those test instances, the sGS-ADMM alone is not efficient enough to solve the problems to the desired accuracy, but the two-phase algorithm that uses a small number of sGS-ADMM iterations to warm-start the pALM is much more efficient.

\begin{figure}[H]
	\centering
	\vspace{-0.4cm}
	\includegraphics[width=6.5in]{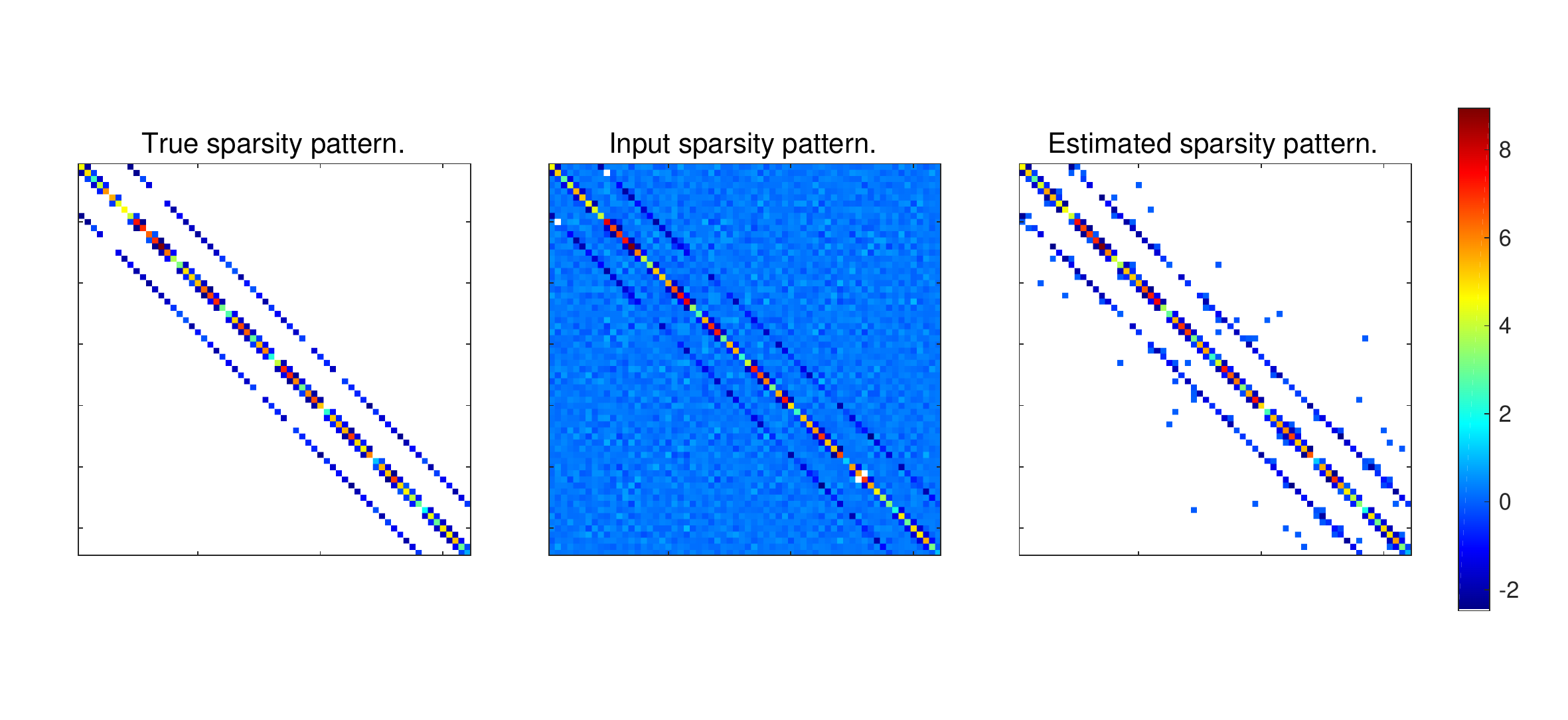}
	\vspace{-1.5cm}
	\caption{The estimated result of grid graph recovery for $n=64$.}
	\label{fig: randgrid_64}
\end{figure}

\begin{figure}[H]
	\vspace{-0.4cm}
	\subfigure[True grid graph.]{\label{fig:randgrid_64_org}
		\includegraphics[width=0.5\linewidth]{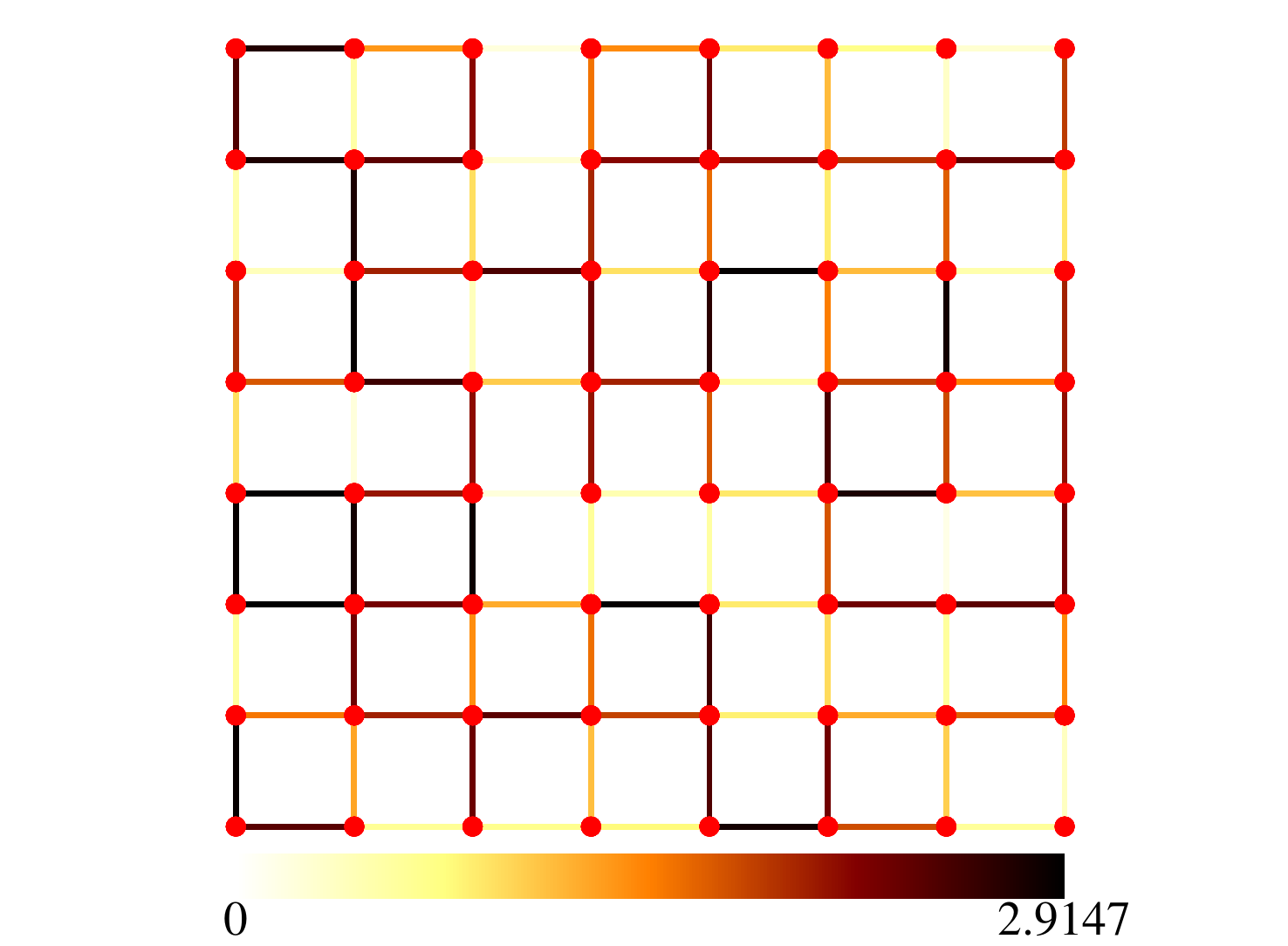}
		\hspace{-1.3cm}}
	\subfigure[Estimated grid graph.]{\label{fig:randgrid_64_recover}
		\includegraphics[width=0.5\linewidth]{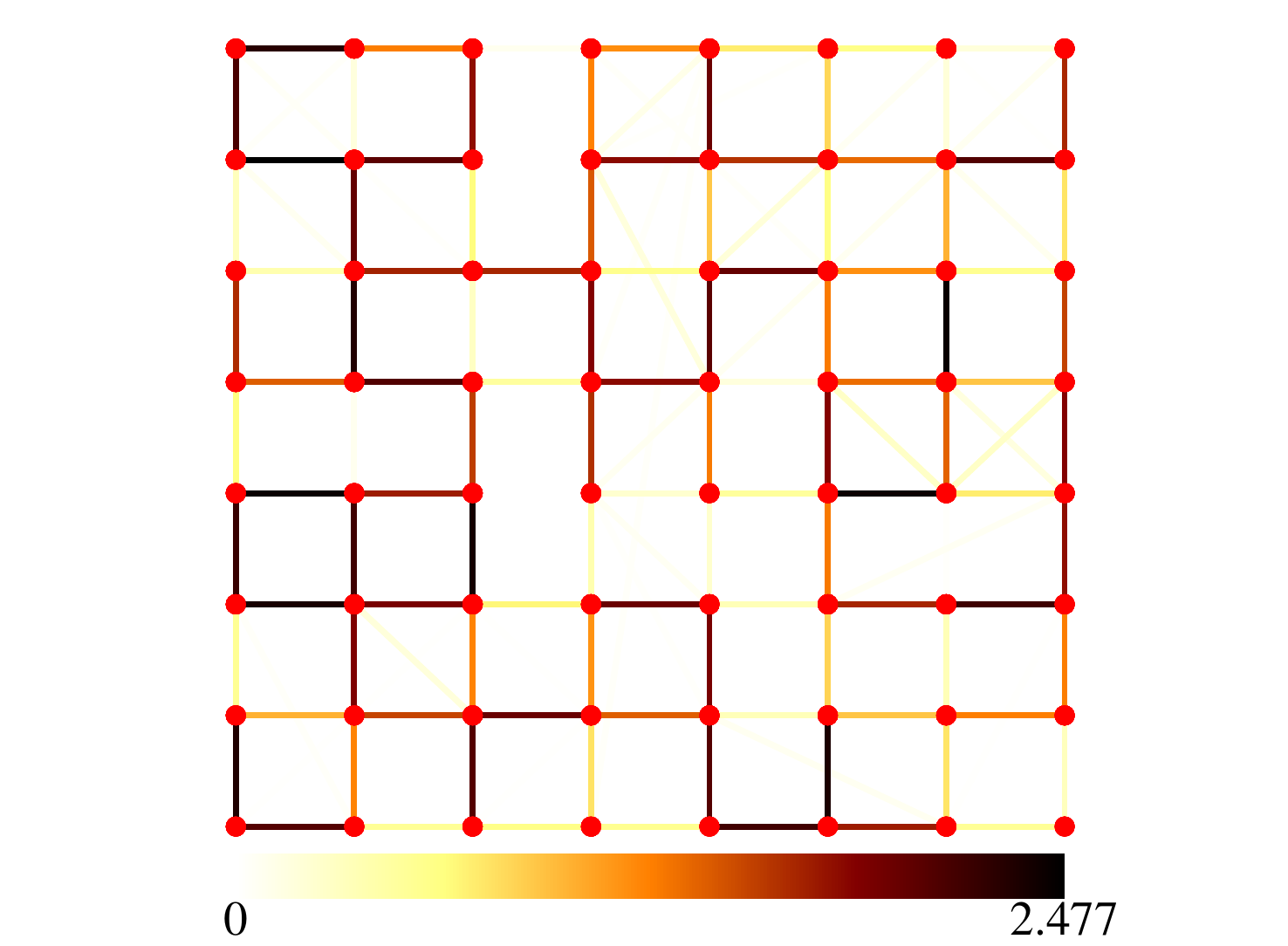}
		\hspace{-1.5cm}}	
	\caption{Visualization of the grid graph recovery for $n=64$.}
	\label{fig:randgrid_64_com}
\end{figure}

{\small
	\begin{ThreePartTable}
		\renewcommand{\multirowsetup}{\centering}
		\renewcommand\arraystretch{1.5}
		\begin{longtable}{ccc}
			\caption{Performance of the proposed model for grid graph recovery.} \label{table_randgrid_acc}\\[-6pt]
			\hline
			\hline
			$(n,m)$ & RE & FS  \\
			\hline
			\endfirsthead
			
			\multicolumn{3}{c}{{ \tablename\ \thetable{} -- continued from previous page}} \\
			
			\hline
			\hline
			$(n,m)$  & RE & FS  \\
			\hline
			\endhead
			
			\multicolumn{3}{r}{{Continued on next page}} \\
			\hline
			\hline
			\endfoot
			
			\hline
			\hline
			\endlastfoot	
			
			(64,952) & 1.47e-1  & 8.43e-1\\
			\hline
			
			(400,39520)  & 1.36e-1 & 6.60e-1\\
			\hline
			
			(900,201405)  & 4.90e-2 & 2.82e-1\\
			\hline
			
			(1600,638040)  & 4.38e-2 & 2.56e-1\\
		\end{longtable}
	\end{ThreePartTable}
}

{\small
	\begin{ThreePartTable}
		\begin{TableNotes}
			\item[\bfseries {\footnotesize Note:} ]  {\footnotesize ``T'' means the two-phase algorithm, ``S" means the sGS-ADMM. ``T(\uppercase\expandafter{\romannumeral1} )'' denotes the iteration number in Phase \uppercase\expandafter{\romannumeral1} and ``4(40)" in ``T(\uppercase\expandafter{\romannumeral2} )'' means ``the pALM iterations (the total inner SSN iterations)".}
		\end{TableNotes}
		\renewcommand{\multirowsetup}{\centering}
		\renewcommand\arraystretch{1.5}
		\begin{longtable}{c|cccc|ccc|cc}
			\caption{Performance of the algorithms for grid graph recovery.} \label{table_randgrid_time}\\[-6pt]
			\hline
			\hline
			$(n,m)$ &  \multicolumn{2}{c}{$\max\{R_{P},R_{D},R_{C}\}$} & \multicolumn{2}{c|}{$R_G$}   & \multicolumn{3}{c|}{ Iteration}  & \multicolumn{2}{c}{Time} \\
            & T & S & T & S & T(\uppercase\expandafter{\romannumeral1} ) & T(\uppercase\expandafter{\romannumeral2} ) & S  & T & S\\
			\hline
			\endfirsthead
			
			\multicolumn{10}{c}{{ \tablename\ \thetable{} -- continued from previous page}} \\
			
			\hline
			\hline
			$(n,m)$ &  \multicolumn{2}{c}{$\max\{R_{P},R_{D},R_{C}\}$} & \multicolumn{2}{c|}{$R_G$}   & \multicolumn{3}{c|}{ Iteration}  & \multicolumn{2}{c}{Time} \\
			& T & S & T & S & T(\uppercase\expandafter{\romannumeral1} ) & T(\uppercase\expandafter{\romannumeral2} ) & S  & T & S\\
			\hline
			\endhead
			
			\multicolumn{10}{r}{{Continued on next page}} \\
			\hline
			\hline
			\endfoot
			
			\hline
			\hline
			\insertTableNotes
			\endlastfoot	
			
			(64,952) & 8.20e-7 & 9.99e-7 & 4.86e-7 & 6.56e-9 & 200 & 4(40) & 4370 & 00:00:02  & 00:00:21\\
			\hline
			
			(400,39520) & 9.59e-7 & 9.99e-7 & 5.28e-6 & 8.50e-8 & 200 & 8(142) & 13868 & 00:01:32  & 00:16:15\\
			\hline
			
			(900,201405) & 6.98e-7 & 9.99e-7 & 7.29e-7 & 1.53e-9 & 200 & 16(300) & 21952 & 00:15:01  & 02:21:45\\
			\hline
			
			(1600,638040) & 9.97e-7 & 9.99e-7 & 1.36e-6 & 1.54e-8 & 200 & 18(342) & 23900 & 01:09:22  & 09:19:00\\
		\end{longtable}
	\end{ThreePartTable}
}

\paragraph{Synthetic dataset \uppercase\expandafter{\romannumeral3}: modular graph recovery.}
We generate a modular graph (also known as a stochastic block graph) ${\cal G}_{\rm modular}(n,n_G,p_1,p_2)$ with $n$ vertices and $n_G$ modules where the vertex attachment probabilities across modules and within modules are $p_1$ and $p_2$, respectively. We take $p_1=0.01$, $p_2=0.3$. The edge weights are randomly selected based on a uniform distribution from $[0.1,3]$. We sample $p=10n$ instances from the multivariate Gaussian distribution ${\cal N}(0,L^{\dagger})$, where $L$ is the associated Laplacian matrix of the graph. We fix $\rho=0.01$ and $k=1$ in \eqref{eq: parameter}.

Figures \ref{fig:randmodular_64} and \ref{fig:randmodular_64_com} show the visualization result of the estimation for the case $(n,n_G)=(64,4)$. Observe that we can get a good estimation of the sparsity pattern and clustering structure. The two metrics for evaluating the performance of the proposed model on various instances of modular graphs are reported in Table \ref{table_randmodular_acc} and the corresponding numerical performance of the two-phase algorithm and the sGS-ADMM is presented in Table \ref{table_randmodular_time}. By comparing the case for $(n,n_G)=(64,4)$ with the result in \cite{kumar2020unified}, our estimation result is a little better in in terms of the F-score.

\begin{figure}[H]
	\centering
	\vspace{-0.2cm}
	\includegraphics[width=6.5in]{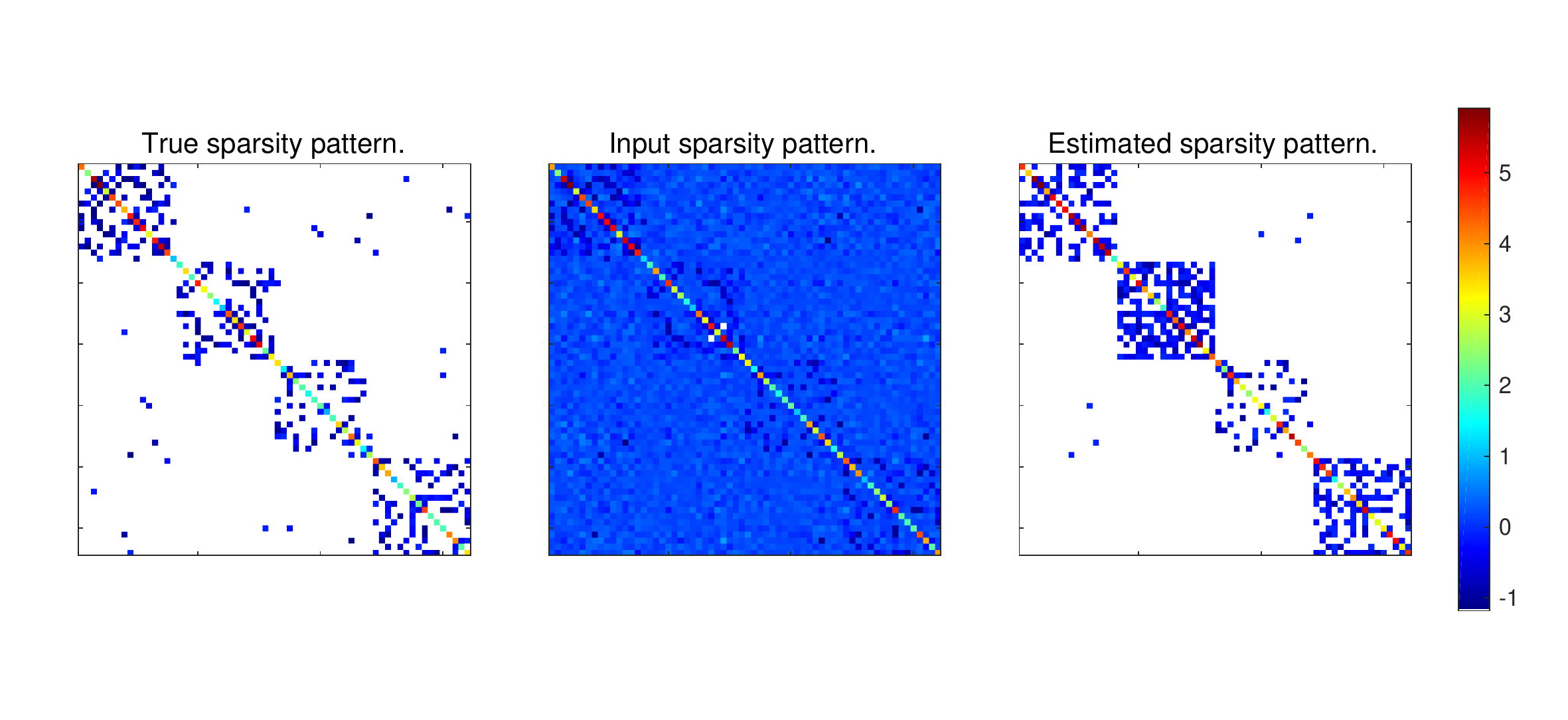}
	\vspace{-1.5cm}
	\caption{The estimated result of modular graph recovery for $(n,n_G)=(64,4)$.}
	\label{fig:randmodular_64}
\end{figure}

\begin{figure}[H]
	\subfigure[True modular graph.]{\label{fig:randmodular_64_org}
		\includegraphics[width=0.5\linewidth]{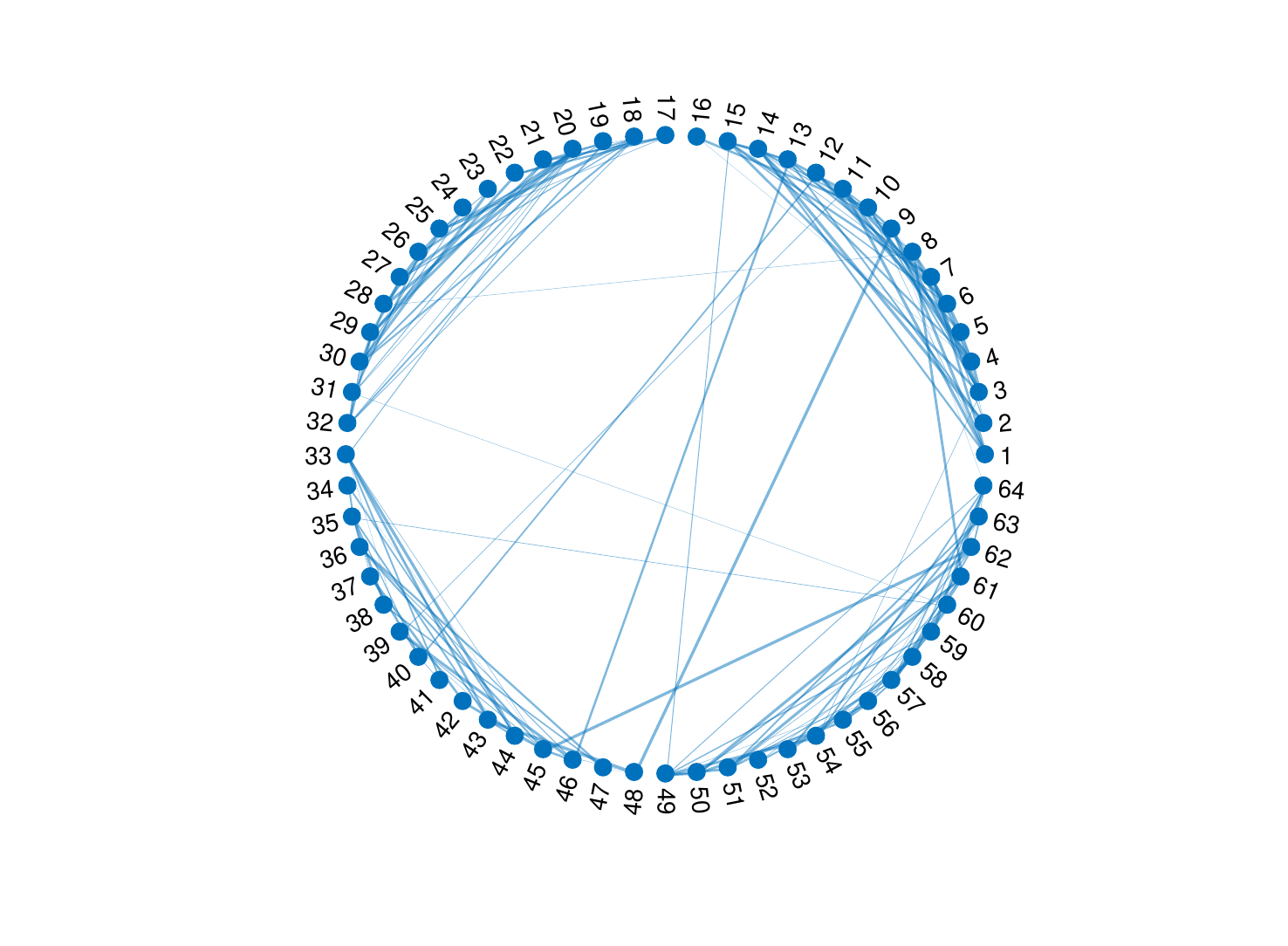}
		\hspace{-1.5cm}}
	\subfigure[Estimated modular graph.]{\label{fig:randmodular_64_recover}
		\includegraphics[width=0.5\linewidth]{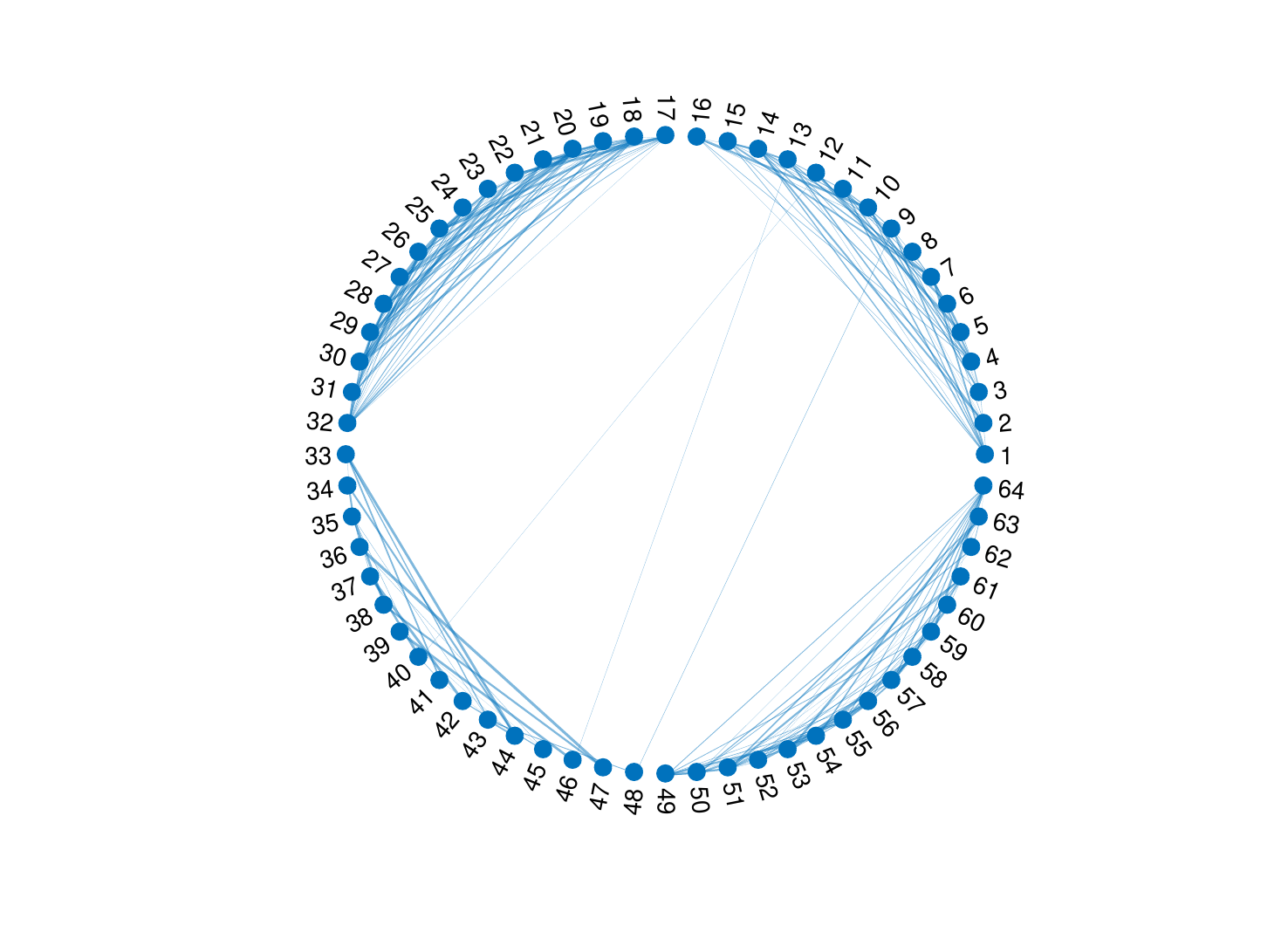}
		\hspace{-1.5cm}}	
	\caption{Visualization of the modular graph recovery for $(n,n_G)=(64,4)$.}
	\label{fig:randmodular_64_com}
\end{figure}

{\small
	\begin{ThreePartTable}
		\renewcommand{\multirowsetup}{\centering}
		\renewcommand\arraystretch{1.5}
		\begin{longtable}{ccc}
			\caption{Performance of the proposed model for modular graph recovery.} \label{table_randmodular_acc}\\[-6pt]
			\hline
			\hline
			$(n,n_G,m)$ & RE & FS  \\
			\hline
			\endfirsthead
			
			\multicolumn{3}{c}{{ \tablename\ \thetable{} -- continued from previous page}} \\
			
			\hline
			\hline
			$(n,n_G,m)$  & RE & FS  \\
			\hline
			\endhead
			
			\multicolumn{3}{r}{{Continued on next page}} \\
			\hline
			\hline
			\endfoot
			
			\hline
			\hline
			\endlastfoot	
			
			(64,4,928) & 1.55e-1  & 8.32e-1\\
			\hline
			
			(400,40,39249)  & 8.98e-2 & 4.76e-1\\
			\hline
			
			(800,80,157620)  & 7.89e-2 & 4.21e-1\\
			\hline
			
			(1000,100,246558)  & 8.40e-2 & 5.58e-1\\
		\end{longtable}
	\end{ThreePartTable}
}

{\small
	\begin{ThreePartTable}
		\renewcommand{\multirowsetup}{\centering}
		\renewcommand\arraystretch{1.5}
		\begin{longtable}{c|cccc|ccc|cc}
			\caption{Performance of the algorithms for modular graph recovery.} \label{table_randmodular_time}\\[-6pt]
			\hline
			\hline
			$(n,n_G,m)$ &  \multicolumn{2}{c}{$\max\{R_{P},R_{D},R_{C}\}$} & \multicolumn{2}{c|}{$R_G$}   & \multicolumn{3}{c|}{ Iteration}  & \multicolumn{2}{c}{Time} \\
			& T & S & T & S & T(\uppercase\expandafter{\romannumeral1} ) & T(\uppercase\expandafter{\romannumeral2} ) & S  & T & S\\
			\hline
			\endfirsthead
			
			\multicolumn{10}{c}{{ \tablename\ \thetable{} -- continued from previous page}} \\
			
			\hline
			\hline
			$(n,n_G,m)$ &  \multicolumn{2}{c}{$\max\{R_{P},R_{D},R_{C}\}$} & \multicolumn{2}{c|}{$R_G$}   & \multicolumn{3}{c|}{ Iteration}  & \multicolumn{2}{c}{Time} \\
			& T & S & T & S & T(\uppercase\expandafter{\romannumeral1}) & T(\uppercase\expandafter{\romannumeral2}) & S  & T & S\\
			\hline
			\endhead
			
			\multicolumn{10}{r}{{Continued on next page}} \\
			\hline
			\hline
			\endfoot
			
			\hline
			\hline
			\endlastfoot	
			
			(64,4,928) & 4.58e-7 & 9.98e-7 & 1.83e-7 & 4.03e-7 & 200 & 3(31) & 2789 & 00:00:01  & 00:00:15\\
			\hline
			
			(400,40,39249) & 4.36e-7 & 9.99e-7 & 1.35e-8 & 3.78e-9 & 200 & 10(190) & 6556 & 00:01:40  & 00:07:39\\
			\hline
			
			(800,80,157620)  & 6.29e-7 & 1.00e-6 & 1.73e-9 & 3.39e-9 & 200 & 11(199) & 3674 & 00:07:10  & 00:18:34\\
			\hline
			
			(1000,100,246558) & 8.68e-7 & 1.00e-6 & 4.58e-9 & 4.33e-8 & 200 & 8(152) & 1843 & 00:08:50  & 00:14:46\\
		\end{longtable}
	\end{ThreePartTable}
}
\subsection{Experiments on real data}
In this subsection, we apply our proposed model on some real data to see how it works on estimating the Gaussian graphical model with sparsity and clustering structure. The visualization is constructed using the software {\tt spectralGraphTopology} \footnote{\tt https://CRAN.R-project.org/package=spectralGraphTopology}.

\paragraph{Real data \uppercase\expandafter{\romannumeral1}: Animals dataset.} We use the Animals
dataset \cite{kemp2008discovery,egilmez2017graph,kumar2020unified} to learn a weighted graph by our model. In the graph, vertices denote animals and edge weights represent the similarities between them. The dataset consists of binary values which are answers to $p=102$ questions for $n=33$ animals. Since the data is the categorical (non-Gaussian) data, we follow the idea in \cite{egilmez2017graph} to compute the input matrix $C$ as summation of the sample covariance matrix and the identity matrix scaled by $1/3$. We aim to find the similarities among the animals. Since the conditional independence pattern is unknown in this real application, we apply the unconstrained model \eqref{uncon_P}. We take $\rho=0.05$ and $k=2$ in \eqref{eq: parameter}. The problem is solved by our two-phase algorithm within $1$ second. The visualization of the estimated graph is presented in Figure \ref{fig:animal}. One can see that the animals are clustered into various meaningful groups. For example, the cluster of animals consisting of Horse, Elephant, etc, are large herbivorous mammals while the cluster of animals consisting of Tiger, Lion, etc, are carnivorous mammals.

\begin{figure}[H]
	\centering
	\vspace{-1.5cm}
	\includegraphics[width=6in]{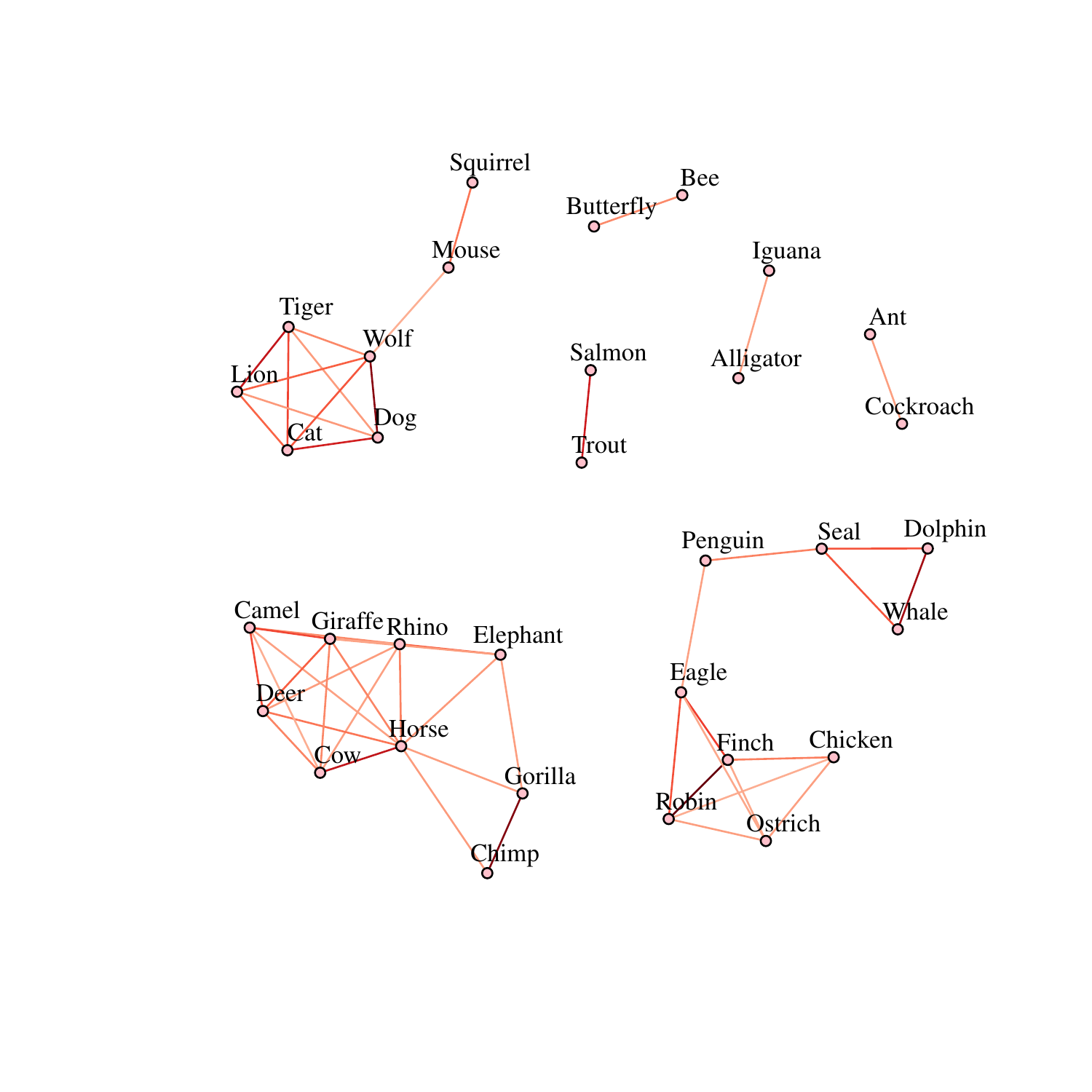}
	\vspace{-2.5cm}
	\caption{Visualization of the estimated result for the Animals dataset.}
	\label{fig:animal}
\end{figure}

\paragraph{Real data \uppercase\expandafter{\romannumeral2}: Cancer genome dataset.} We consider the RNA-Seq Cancer Genome Atlas Research Network \cite{weinstein2013cancer,kumar2020unified}. In the dataset, there are $n=801$ labeled samples, and each of them has $p=20531$ features. The dataset consists of five types of cancer, which are labeled with colors in the figure: black, blue, red, violet and green, respectively. Our goal is to cluster the samples based on the given features assuming that we do not known the true labels. We apply the unconstrained model \eqref{uncon_P} and take $\rho=0.1$, $k=2$ in \eqref{eq: parameter}. The problem is solved by the two-phase algorithm in $48$ seconds. Figure \ref{fig:cancer} presents the visualization of the estimated result. One can see that the samples are clustered into five groups except for about thirty isolated samples. The clustering result is consistent with the label information and the samples in different groups are completely separated.

\begin{figure}[H]
	\centering
	\vspace{-1.5cm}
	\includegraphics[width=6in]{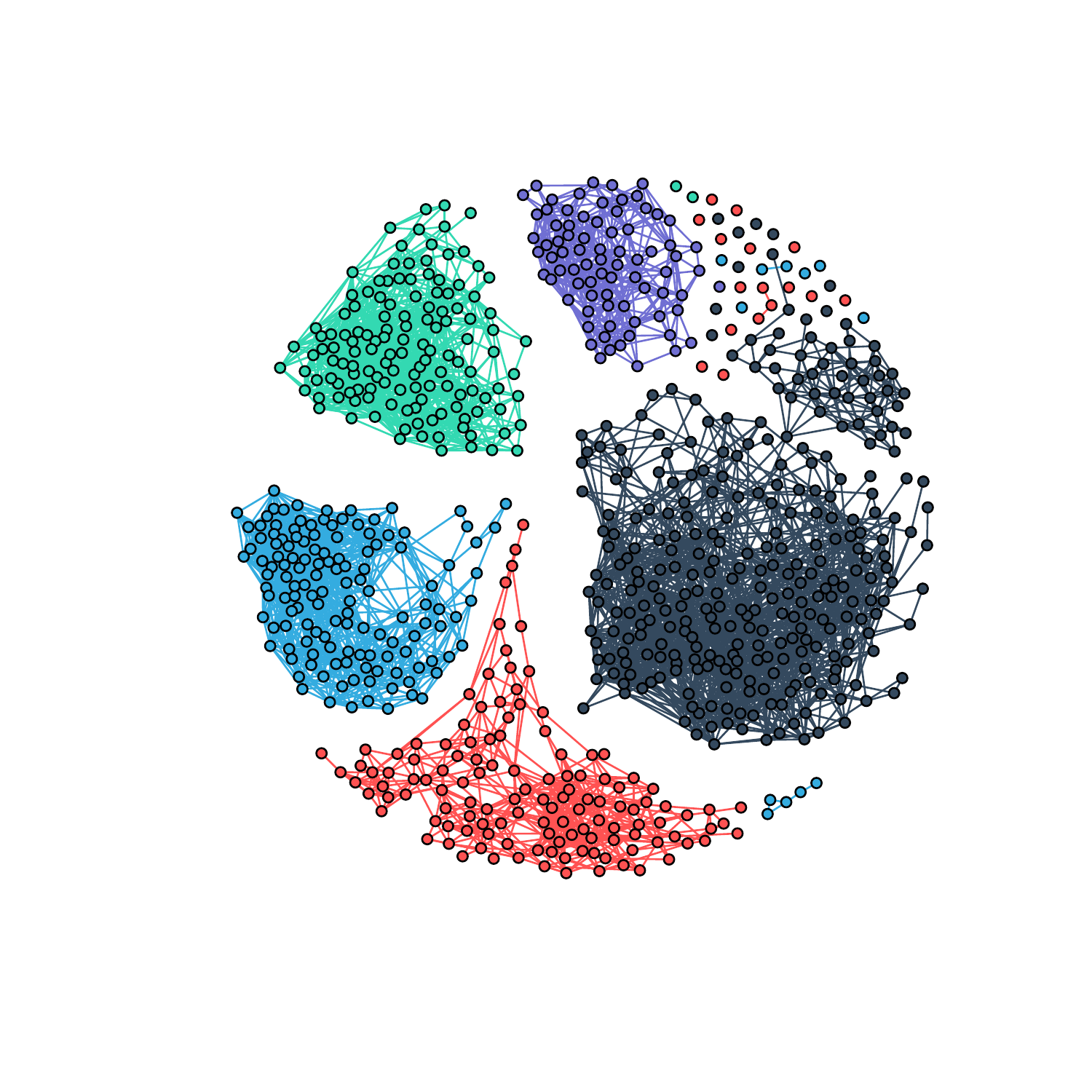}
	\vspace{-2.5cm}
	\caption{Visualization of the estimated result for the Cancer genome dataset.}
	\label{fig:cancer}
\end{figure}

\paragraph{Real data \uppercase\expandafter{\romannumeral3}: Zoo dataset.} We consider the Zoo dataset from the UCI Machine Learning Repository, which contains $n=100$ animals and each animal has $p=17$ Boolean-valued attributes. The dataset contains seven types of animals which are known.
To be specific, the set contains $41$ kinds of mammals, $20$ kinds of birds, $5$ kinds of reptiles, $13$ kinds of fish, $3$ kinds of amphibians, $8$ kinds of bugs and $10$ invertebrates. Each type is labeled in the figure by a different color: black, violet, red, green, blue, yellow and pink, respectively. Since the data is the categorical data, we use the same technique as the case for the Animals dataset, that is, computing the input matrix $C$ as summation of the sample covariance matrix and the identity matrix scaled by $1/3$. In the experiment we take $\rho=0.05$ and $k=2$ in \eqref{eq: parameter}. The problem is solved by the two-phase algorithm within $1$ second. We compare the clustering result of the model \eqref{uncon_P} with the true groups in Figure \ref{fig:zoo}. As one can see, the animals belonging to each group are clustered together except for the reptiles. Due to the small sample size in this dataset, there exist some wrong connections across different clusters, which are indicated by the grey colored edges in the figure. Some of the wrong connections are consistent with our usual expectation. For example, there exists an edge between platypus and penguin since they are both vertebrate warm blooded animals that lay eggs. Note that the animals belonging to the relatively large groups: mammals, birds and fish, are clearly separated. In addition, the cluster consisting of mammals is further divided into three sub groups: the carnivorous mammals like lion, the large herbivorous mammals like elephant, and the small herbivorous mammals like squirrel.

\begin{figure}[H]
	\centering
	\vspace{-1.5cm}
    \hspace{-1.5cm}
	\includegraphics[width=7in]{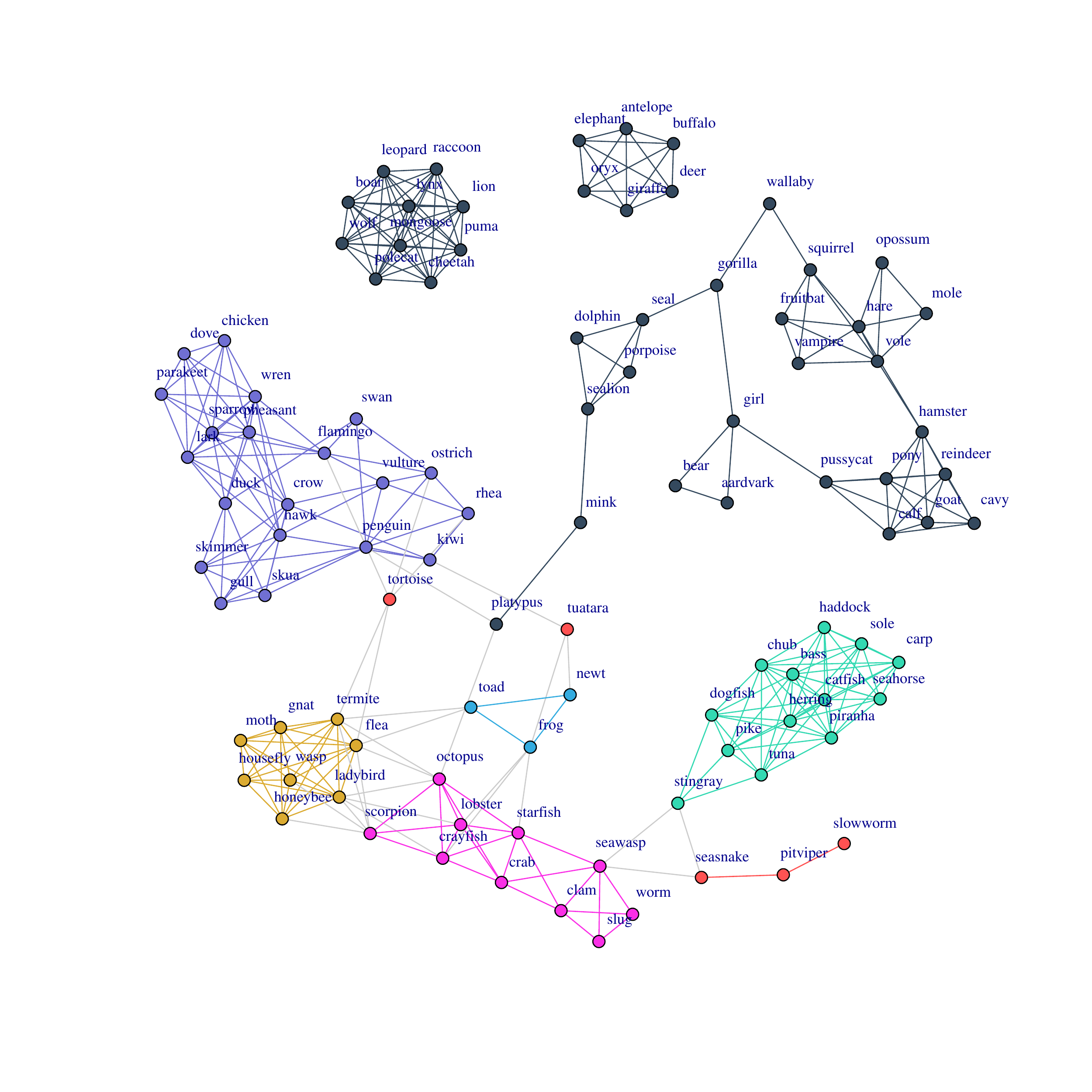}
	\vspace{-2.5cm}
	\caption{Visualization of the estimated result for the Zoo dataset.}
	\label{fig:zoo}
\end{figure}

\section{Conclusion}
\label{sec: conclusion}
In this paper, we propose a new model to learn the sparsity and hidden clustering structure in a Gaussian graphical model. In addition, we design an efficient two-phase algorithm to solve the underlying large scale convex optimization to high accuracy. Specifically,  we design the sGS-ADMM in the first phase to generate an initial point to warm-start the second phase of the pALM, where each of its subproblems is solved by the semismooth Newton method. Numerical experiments on both synthetic data and real data have demonstrated the good performance of our model, and the efficiency and robustness of our proposed algorithm.


\end{document}